\documentclass[reqno]{amsart}

\usepackage{amsmath,amssymb,amsthm,amsfonts}
\usepackage{color}
\usepackage{hyperref}
\usepackage{enumerate}
\usepackage{tikz} 
\usetikzlibrary{arrows} 

\makeatletter
\@addtoreset{equation}{section}
\makeatother

\theoremstyle{plain}
\newtheorem{prop}{Proposition}[section]
\newtheorem{theorem}[prop]{Theorem}

\newtheorem{lemma}[prop]{Lemma}
\newtheorem{definition}[prop]{Definition}
\newtheorem{cor}[prop]{Corollary}

\theoremstyle{remark}
\newtheorem{remark}[prop]{Remark}

\newcommand{\N}{\mathbb{N}}
\newcommand{\R}{\mathbb{R}}

\renewcommand{\P}{\mathbb{P}}
\newcommand{\E}{\mathbb{E}}

\newcommand{\dd}{\mathrm{d}} 
\newcommand{\eps}{\varepsilon}

\def\tire{\thinspace--\thinspace}

\renewcommand{\Cap}{\mathrm{cap}\,}  

\newcommand{\be}{\begin{equation}}
\newcommand{\ee}{\end{equation}}                  

\title{Berman-Konsowa principle for reversible Markov jump processes}

\author{Frank den Hollander}
\address{Mathematical Institute, Leiden University, P.O.\ Box 9512,
2300 RA Leiden, The Netherlands}
\email{denholla@math.leidenuniv.nl}

\author{Sabine Jansen}
\address{Fakult\"at f\"ur Mathematik, Ruhr-Universit\"at Bochum,
 44780 Bochum, Germany}
\email{sabine.jansen@ruhr-uni-bochum.de}

\date{October 13, 2014}

\begin{document}

\begin{abstract}
In this paper we prove a version of the Berman-Konsowa principle for reversible 
Markov jump processes on Polish spaces. The Berman-Konsowa principle 
provides a variational formula for the capacity of a pair of disjoint measurable 
sets. There are two versions, one involving a class of probability measures for 
random finite paths from one set to the other, the other involving a class of finite 
unit flows from one set to the other. The Berman-Konsowa principle complements 
the Dirichlet principle and the Thomson principle, and turns out to be especially 
useful for obtaining sharp estimates on crossover times in metastable interacting 
particle systems.
  
\medskip\noindent 
\emph{MSC 2010}: 
60J05, 
60J25, 
60J45, 
60J75. 

\medskip\noindent
\emph{Keywords}: Reversible Markov jump processes, electric networks, potential 
theory, capacity, Dirichlet principle, Thomson principle, Berman-Konsowa principle.\\
\emph{Acknowledgment}: The authors are grateful to Anton Bovier for fruitful input.
The research in this paper was supported by ERC Advanced Grant 267356 VARIS
and was carried out while SJ held a postdoctoral position in Leiden from September 
2012 until May 2013.

\end{abstract}

\maketitle


\section{Introduction} 
\label{s:intro}

Section~\ref{ss:motivation} provides the motivation, Section~\ref{ss:setting} formulates the 
setting, Section~\ref{ss:theorems} states the main theorems, while Section~\ref{ss:discussion}
discusses these theorems and places them in their proper context.

\subsection{Motivation}
\label{ss:motivation}

The motivation for the present paper comes from the theory of \emph{metastability} for interacting 
particle systems, i.e., systems consisting of a large number of interacting random components 
evolving according to a Markovian random dynamics on a space of configurations. As time evolves, 
the system moves through different subregions of its configuration space, corresponding to different 
``thermodynamic phases''. Typically, in a metastable setting, on short time scales the system reaches 
a quasi-equilibrium inside a single subregion, while on long time scales it makes rapid transitions 
between different subregions, with crossover times that are exponentially distributed on the scale 
of their mean. The task of mathematics is to analyze such systems in detail, and to explain the 
experimentally observed universality in their metastable behavior. This is a conceptual program 
of great challenge.

There are two main approaches to metastability: (1) the \emph{pathwise approach}, initiated 
by Freidlin and Wentzell~\cite{freidlin-wentzell}, in which a detailed description is given of the 
trajectories of the system, and the focus is on identifying the most likely trajectories and to estimate 
their probabilities; (2) the \emph{potential-theoretic approach}, initiated by Bovier, Eckhoff, 
Gayrard and Klein~\cite{bovier-eckhoff-gayrard-klein}, in which metastability is viewed as a 
sequence of visits of the trajectory to different metastable sets, and the focus is on a precise 
analysis of the respective hitting probabilities and hitting times of these sets with the help of 
potential theory. Phrased differently, the problem of understanding the metastable behavior 
of Markov processes is translated to the study of equilibrium potentials and capacities of 
\emph{electric networks} (Doyle and Snell~\cite{doyle-snell}).

More precisely, the configurations of the system are viewed as the vertices of the network 
and the transitions between pairs of configurations as the edges of the network. The transition 
probabilities are represented by the conductances associated with the edges. In this language, 
the hitting probability of a target set of configurations as a function of the starting configuration 
can be expressed in terms of the equilibrium potential on the network when the potential is set
equal to one on the vertices of the target set and equal to zero on the starting vertex. The average 
hitting time of the target set can then be expressed in terms of the equilibrium potential and the 
capacity associated with the target set and the starting vertex. For metastable sets it turns out 
that \emph{the average hitting time is essentially the inverse of the capacity}.

A key observation in the potential-theoretic approach is the fact that capacities can be estimated 
by exploiting powerful \emph{variational principles}. In fact, dual variational principles are available 
that express the capacity both as an infimum over potentials (\emph{Dirichlet principle}) and as 
a supremum over flows (\emph{Thomson principle}). This opens up the possibility to derive sharp 
upper bounds and lower bounds on the capacity via a judicious choice of test functions. In fact, 
with the proper physical insight, test functions can be found for which the upper bound and the 
lower bound are asymptotically equivalent (in an appropriate limit corresponding to a metastable 
regime). Consequently, with the help of the potential-theoretic approach asymptotic estimates of 
the average crossover time can be derived that are much sharper than those typically obtainable 
with the help of the pathwise approach.

Both the Dirichlet principle and the Thomson principle have been key tools in electric network 
theory for many years (Doyle and Snell~\cite{doyle-snell}). More recently, Berman and 
Konsowa~\cite{berman-konsowa} proved two variational formulas for \emph{finite} electric networks, 
one in terms of probability measures on paths from one set to another (or their associated flows),  
the other in terms of random cuts (or their associated co-boundaries). The Berman\tire Konsowa 
principle has been instrumental in obtaining sharp bounds for interacting particle systems with 
complex interactions (Bovier, den Hollander and Nardi~\cite{bovier-denhollander-nardi}, 
Bianchi, Bovier and Ioffe~\cite{bianchi-bovier-ioffe}, Bovier, den Hollander and
 Spitoni~\cite{bovier-denhollander-spitoni}). 
 
The goal of the present paper it to \emph{generalize} the Berman\tire Konsowa principle for reversible 
Markov chains on \emph{finite spaces} to reversible Markov jump processes on \emph{Polish 
spaces}. The principal technical difficulty lies in the richer recurrence behavior of the Markov 
jump process. Consequently, the generalization comes in two steps. First we treat the case of 
finite total conductance, corresponding to continuous-time Markov processes whose underlying 
discrete-time jump chain is positive recurrent. After that we treat the case of infinite total 
conductance, associated with null recurrent or transient jump chains. The latter requires a 
preliminary re-examination of the relevant potential theory, the solution of the Dirichlet problem 
being no longer unique (Appendix~\ref{app1}), followed by a careful truncation procedure that 
approximates infinite conductances by finite conductances. 

Our principal motivation is an application to metastability for continuum interacting particle systems 
(den Hollander and Jansen~\cite{denhollander-jansen}, den Hollander, Jansen, Koteck\'y and 
Pulvirenti~\cite{denhollander-jansen-kotecky-pulvirenti}). For an overview on the potential-theoretic 
approach to metastability we refer the reader to the monograph by Bovier and den 
Hollander~\cite{bovier-denhollander}.

\subsection{Setting}
\label{ss:setting}

Let $(\Omega,\mathcal{F})$ be a Polish space. Let $X=(X_t)_{t \geq 0}$ be a continuous-time 
irreducible \emph{recurrent} Markov jump process on $\Omega$ with transition rates 
$k(x,\dd y)$ and reversible invariant measure $\mu(\dd x)$ (Stroock~\cite{stroock}), i.e., 
$k(x,\dd y)$ is the rate to jump from $x$ to a neighborhood $\dd y$ of $y$, and 
\be 
\label{reversible}
\mu( \dd x) k(x,\dd y) = \mu(\dd y) k (y,\dd x).
\ee
Define
\be
\label{eq:KPk}
K(\dd x, \dd y) = \mu( \dd x) k(x,\dd y),
\ee 
which is a non-negative symmetric measure on $\Omega \times \Omega$. We assume that 
$K(\Omega \times \Omega) < \infty$ and that $k(x,\Omega)>0$ for $\mu$-a.a.\ $x$. 
For measurable $C,D \subset \Omega$, we think of $K(C \times D) = K(D \times C )$ as the 
\emph{conductance} of $X$ between $C$ and $D$. As long as $X_0$ is drawn from a probability 
measure that is absolutely continuous w.r.t.\ $\mu$, $X$ makes finitely jumps in every finite 
time-interval (see Appendix~\ref{app1}).

For $f\colon\,\Omega \to \R$ measurable and bounded, define 
\be
\label{eq:generator}
(Lf)(x) = \int_\Omega \bigl[ f(y) - f(x) \bigr] k(x,\dd y)
\ee
and
\be
\label{eq:dirichlet-form}
\mathcal{E}(f) =  \int_{\Omega} \bigl[ - (Lf)(x) \bigr] f(x) \mu(\dd x)
= \frac{1}{2}  \int_{\Omega\times \Omega} \bigl[ f(y) - f(x) \bigr]^2 K(\dd x,\dd y),
\ee
which are the infinitesimal generator (with domain $\mathcal{D}(L)$) and the 
Dirichlet form (with domain $\mathcal{D}(\mathcal E)$) associated with $X$ 
(see Appendix~\ref{app1} for technical details). 

Later we will exploit the fact that $X$ can be constructed as a random time-change 
of a discrete-time Markov chain $Z=(Z_n)_{n\in\N}$ with transition kernel
\be \label{eq:discrete} 
\frac{k(x,\dd y)}{k(x,\Omega)}.
\ee
The finiteness of $K$ ($K(\Omega \times \Omega) < \infty$) is equivalent to $Z$ 
being \emph{positive recurrent}. It is possible that $Z$ is positive recurrent while 
$X$ is null-recurrent ($\mu(\Omega) = \infty$). See, in particular, the comments 
after Proposition~\ref{prop:non-explosion2}. At the end of Section~\ref{ss:theorems} 
the case of \emph{recurrent} $X$ with \emph{null-recurrent} $Z$ and the case of 
\emph{transient} $X$ will be included, i.e., the extension to $K(\Omega\times \Omega) 
= \infty$ will be made.

Throughout the sequel, $A,B \subset \Omega$ are \emph{disjoint} and measurable.
The harmonic function and the capacity of the pair $(A,B)$ associated with $X$ are 
defined as
\be 
\label{eq:habdef}
h_{AB}(x) = \left\{\begin{array}{ll}
\P_x(\tau_A < \tau_B), &x \in (A \cup B)^c,\\
1, &x \in A,\\
0, &x \in B,
\end{array}
\right.
\ee
and
\be 
\label{eq:cap-def1}
\Cap(A,B) = \int_A (-Lh_{AB})(x)\,\mu(\dd x),
\ee
where $\tau_C = \inf\{t > 0\colon\,X_t \in C\}$ is the first hitting time of $C$, $\P_x$ is the law 
of $X$ given $X_0=x$, and $(-Lh_{AB})(x)$ is the \emph{equilibrium charge} at $x\in A$. The 
\emph{Dirichlet principle} says that the capacity satisfies the variational formula (Fukushima, Oshima and 
Takeda~\cite{fukushima-oshima-masayoshi})
\be
\label{capdef} 
\Cap(A,B)= \inf_{h \in \mathcal{V}_{AB}} \mathcal{E}(h)
\ee
with
\be
\label{VABdef}
\mathcal{V}_{AB} = \bigl\{ h \colon\, \Omega \to \R \mid h|_A =1,\ h|_{B} = 0,\,
0 \leq h \leq 1\,\,\mu\text{-a.e.},\,\mathcal{E}(h)<\infty\bigr\},
\ee
and has $h=h_{AB}$ as its unique minimizer. For completeness the proof is given in 
Appendix~\ref{app1}.


\subsection{Theorems}
\label{ss:theorems}

Let $\Gamma_{AB}$ be the set of finite paths from $A$ to $B$, i.e., 
\be 
\Gamma_{AB} = \bigcup_{n\in \N} \bigl\{\gamma = (\gamma_0,\ldots,\gamma_n) 
\mid \gamma_0 \in A,\ \gamma_n \in B,\ \gamma_i \notin B \text{ for } i =1,\ldots,n-1\bigr\}. 
\ee
Write $(x,y) \in \gamma$ when there is a $j\in\N$ such that $(x,y)=(\gamma_{j-1},\gamma_{j})$.
Let $\P$ be a probability measure on $\Gamma_{AB}$. Then there is a unique measure $\Phi 
= \Phi_\P$ on $\Omega \times \Omega$ such that for every bounded and measurable 
$f\colon\,\Omega \times \Omega \to [0,\infty)$, 
\be  
\label{eq:edge-proba}
 \E \left[ \sum_{(x,y) \in\gamma} f(x,y) \right] 
 = \int_{\Omega\times \Omega} f(x,y)\Phi(\dd x,\dd y),
\ee
where $\gamma$ is the random element of $\Gamma_{AB}$ whose law is $\P$. Picking $f(x,y) = 
\mathbf{1}_{C\times D}(x,y)$, $C,D \subset \Omega$, we see that $\Phi(C \times D)$ is the expected 
number of edges $(x,y)$ in the random path $\gamma$ with $x \in C$ and $y \in D$. In particular, 
$\Phi$ is finite if and only if the expected length of $\gamma$ is finite.

Our first main result is the following theorem.
 
\begin{theorem}[Berman-Konsowa principle: path version]
\label{thm:BK1}
Let $\mathcal{P}_{AB}^K$ be the set of probability measures $\P$ on $\Gamma_{AB}$ such that:\\ 
{\rm (a)} $\Phi_\P$ is absolutely continuous with respect to $K$.\\
{\rm (b)} There is a measurable $\chi \subset \Omega \times \Omega$ such that $\Phi_\P (\chi^{\rm c}) 
= 0$ and $(y,x) \notin \chi$ for all $(x,y) \in \chi$.\\
Then 
\be
\label{eq:Cap1}
\Cap(A,B)  = \sup_{\P \in \mathcal{P}_{AB}^K} \E\left[\left( \sum_{(x,y) \in \gamma} 
\frac{\dd \Phi_{\P}}{\dd K}(x,y) \right)^{-1} \right]. 
\ee
This identity remains true when the supremum is restricted to the smaller set of probability measures 
on finite self-avoiding paths from $A$ to $B$. 
\end{theorem}

\begin{remark} 
On finite state spaces, the variational formula in~\eqref{eq:Cap1} also holds for measures $\P$ that do 
not satisfy condition (b). This has been proven by Berman and Konsowa~\cite{berman-konsowa} with
the help of a simultaneous use of random paths and random cuts, in contrast to our proof, which uses 
random paths only.
\end{remark} 

In order to state our second main result we need the following definition of a
flow along edges.

\begin{definition}
\label{def:flow}
A \emph{unit flow} from $A$ to $B$ is a sigma-finite measure $\Phi$ on $\Omega 
\times \Omega$ such that 
\begin{enumerate} [(1)]
\item 
$\Phi(A \times \Omega)=1$ and $\Phi(\Omega \times A)=0$.
\label{aflow}
\item 
$\Phi(\Omega \times B)= 1$ and $\Phi(B \times \Omega) = 0$. 
\label{bflow}
\item 
$\Phi(\Omega \times C) = \Phi(C \times \Omega)$ for all measurable $C\subset \Omega 
\backslash (A\cup B)$. 
\label{kirchhoff}
\item 
There is a measurable $\chi \subset \Omega \times \Omega$ such that $\Phi(\chi^\mathrm{c}) 
= 0$ and  $(y,x) \notin \chi$ for all $(x,y) \in \chi$. 
\label{directed}
\end{enumerate}
\end{definition}

\noindent
Conditions (\ref{aflow})--(\ref{bflow}) say that the total flow out of $A$ and into $B$ is $1$
while the total flow into $A$ and out of $B$ is $0$. Condition (\ref{kirchhoff}) says that 
the flow is \emph{divergence-free} in $\Omega \backslash (A \cup B)$, which we refer
to as \emph{Kirchhoff's law} (the terminology used for discrete spaces). Condition 
(\ref{directed}) says that an edge and its reverse cannot not lie in the support of $\Phi$
simultaneously, i.e., the flow is \emph{oriented}.

The flow is called \emph{loop-free} when the set $\chi$ in condition (\ref{directed}) can be 
chosen such that it contains no loops, i.e., if $(\gamma_0,\ldots,\gamma_n)$ is a finite 
sequence with $\gamma_n=\gamma_0$, then $(\gamma_{j-1},\gamma_j) \notin \chi$ 
for some $j = 1,\ldots,n$.

Let $\Phi$ be a unit flow from $A$ to $B$, let $\nu(C)= \Phi(C \times \Omega)$ be its left 
marginal, and let $\ell(x,\dd y)$ be any probability transition kernel such that $\Phi(\dd x, 
\dd y) = \nu(\dd x) \ell(x, \dd y)$. Let $Y=(Y_n)_{n \in \N_0}$ be the Markov chain with 
initial law $\P^\Phi(Y_0 \in C) = \nu(A \cap C)$, $C \subset \Omega$ measurable, 
and probability transition kernel $\ell$, and put $\tau_B^Y= \min\{n\in \N \mid Y_n \in B\}$. In 
Section~\ref{s:Proofs} we will show that the law of $Y_n$ conditioned on $\tau_B^Y>n$ 
is absolutely continuous w.r.t.\ $\nu$. Therefore changes of $\ell$ on $\nu$-null sets do 
not affect the law of $Y$ stopped in $B$, i.e., the law of the stopped process is uniquely 
determined by the flow $\Phi$. We also show that if $\Phi$ is finite, i.e., $\Phi(\Omega 
\times \Omega)<\infty$, then $\E^\Phi[ \tau_B^Y] <\infty$, i.e., $Y$ is positive recurrent. 
The latter implies that $\tau_B^Y<\infty$ $\P^\Phi$-a.s.

\begin{theorem}[Berman\tire Konsowa principle: flow version] 
\label{thm:BK2}	
Let $\mathcal{U}_{AB}^K$ be the set of unit flows from $A$ to $B$ that are absolutely 
continuous w.r.t.\ $K$ and finite. Then
\be
\label{eq:Cap2}
\Cap(A,B) = \sup_{\Phi \in \mathcal{U}_{AB}^K} 
\E^\Phi\left[ \left( \sum_{n=1}^{\tau_B^Y} \frac{\dd \Phi}{\dd K}(Y_{n-1},Y_n) \right)^{-1} \right]. 
\ee
The identity remains true when the supremum is restricted to the smaller set of 
loop-free unit flows.
\end{theorem}

An important example of a loop-free unit flow is the \emph{harmonic flow} defined 
by
\be 
\label{eq:harmonic-flow}
\Phi_{AB}(\dd x, \dd y) = \frac{1}{\Cap(A,B)} 
\bigl[ h_{AB}(x)-h_{AB}(y) \bigr]_+ K(\dd x, \dd y).  
\ee
We will show that this flow has finite self-avoiding paths and is a maximizer of 
\eqref{eq:Cap2}. It is not necessarily the unique maximizer. Any $\P$ such that
$\Phi_\P=\Phi_{AB}$ is a maximizer of \eqref{eq:Cap1}. 

We close with the statement that the assumption of positive recurrence of $Z$, 
which was made below \eqref{eq:discrete}, can be dropped. We only require 
that $k(x,\dd y)$ admits a sigma-finite measure $\mu(\dd x)$ satisfying \eqref{reversible}, 
and do allow for $K(\Omega\times \Omega)=\infty$.

\begin{theorem}
\label{thm:BK12ext}
Suppose that $K((A \cup B) \times \Omega)<\infty$. Then Theorems~{\rm \ref{thm:BK1}} 
and {\rm \ref{thm:BK2}} extend to $K(\Omega\times \Omega) = \infty$, i.e., to recurrent 
$X$ with null-recurrent $Z$ and to transient $X$.  
\end{theorem}

For transient $X$, capacity can be \emph{defined} by \eqref{capdef}, $h_{AB}$ can be 
\emph{defined} as the unique minimizer of \eqref{capdef}, and \eqref{eq:cap-def1} can 
be \emph{shown} to hold (see Appendix~\ref{app1}). However, in general no explicit 
expression is available for $h_{AB}$ in terms of hitting times, as in \eqref{eq:habdef}
for recurrent $X$. In fact, the natural analogue of \eqref{eq:habdef}, namely, the function 
$g_{AB}$ given by $$g_{AB}(x) = \P_x(\tau_A<\tau_B, \tau_A<\infty),$$ is \emph{not} the 
minimizer of \eqref{capdef} (see Lemma~\ref{lem:harmonic-function} below).


\subsection{Discussion}
\label{ss:discussion}

We place the results from Section~\ref{ss:theorems} in their proper context.

\bigskip\noindent
{\bf 1.}
We have shown that the \emph{Berman\tire Konsowa principle} holds for reversible Markov 
jump process on general Polish spaces. It provides dual variational formulas for the 
capacity, the first running over probability measures $\P$ on the set of finite paths connecting 
$A$ and $B$, the second running over unit flows $\Phi$ from $A$ to $B$. Each probability 
measure $\P$ gives rise to a unit flow $\Phi=\Phi_\P$. Conversely, each unit flow $\Phi$ 
gives rise to a path measure $\P=\P_\Phi$, though not uniquely.
  
\bigskip\noindent
{\bf 2.} 
The Berman\tire Konsowa principle complements the \emph{Dirichlet principle} in \eqref{capdef},
and also the \emph{Thomson principle}, which says that
\be
\frac{1}{\Cap(A,B)} = \inf_{\Phi \in \mathcal{U}_{AB}^K} 
\int_{\Omega \times \Omega} \Bigl( \frac{\dd \Phi}{\dd K} \Bigr)^2 \dd K
\ee
with the harmonic flow from \eqref{eq:harmonic-flow} as the unique minimizer. For 
completeness, the proof of the Thomson principle is given in Appendix~\ref{app1}. 

\bigskip\noindent
{\bf 3.} 
Just as for finite state spaces, the Berman\tire Konsowa bound improves the Thomson bound. Indeed, 
for any $\P \in \mathcal{P}_{AB}^K$ we have, by (\ref{eq:edge-proba})--(\ref{eq:Cap1}) and Jensen's 
inequality,
\be
\begin{aligned}
\Cap(A,B) 
& \geq \E\left[ \left( \sum_{(x,y) \in \gamma} \frac{\dd \Phi_\P}{\dd K}(x,y)\right)^{-1} \right]
\geq \left( \E \left[ \sum_{(x,y) \in \gamma} \frac{\dd \Phi_\P}{\dd K}(x,y)\right] \right)^{-1} \\
&= \left(\ \int_{\Omega \times \Omega} \frac{\dd \Phi_\P}{\dd K}\, \dd \Phi_\P \right)^{-1}
= \left(\ \int_{\Omega \times \Omega} \left(\frac{\dd \Phi_\P}{\dd K}\right)^2 \dd K \right)^{-1}. 
\end{aligned}
\ee

\medskip\noindent
{\bf 4.}
The Dirichlet principle and the Thomson principle complement each other: upper bounds on 
capacities can be obtained by choosing test potentials, lower bounds by choosing test unit 
flows. The Berman\tire Konsowa principle is stronger than the Thomson principle in that it leads 
to better bounds, even though the suprema are the same. This is particularly helpful for 
obtaining \emph{approximations} of capacities. 
 
\medskip\noindent
{\bf 5.} 
In order to derive approximations of capacities it is possible to work with ``leaky flows'' instead
of unit flows, i.e., flows for which the condition ``flow out of $A$ = flow into $B$ = 1'' is fulfilled
with a small error. Indeed, it is possible to quantify the discrepancy between suprema for leaky 
flows and suprema for unit flows in terms of this error, and this allows for greater \emph{flexibility} 
in the approximation procedure. We refer the reader to the monograph by Bovier and den 
Hollander~\cite{bovier-denhollander} for further details.
  
\medskip
The remainder of this paper is organized as follows. In Section~\ref{s:Proofs} we give 
the proof of Theorems~\ref{thm:BK1} and \ref{thm:BK2}, and their extension in 
Theorem~\ref{thm:BK12ext}. In Appendix~\ref{app1} we list some technical facts 
about Dirichlet forms, give the proof of the Dirichlet principle and the Thomson 
principle in the general setting considered in this paper, and show that capacities
can be approximated via truncation. In Appendix~\ref{app2} we give the interpretation 
of the three variational principles for finite electric networks. The two appendices take 
up about half of the paper and rely on basic results from the literature.  


\section{Proofs}
\label{s:Proofs}

Sections~\ref{ss:prepprop} and \ref{ss:discrlem} state and prove a proposition and 
a lemma that are needed in the proofs of Theorems~\ref{thm:BK1} and \ref{thm:BK2} 
in Section~\ref{ss:prooftheorems}. Section~\ref{ss:harmflow} looks at the harmonic flow.
Throughout this section, $A$ and $B$ are fixed disjoint measurable subsets of $\Omega$. 


\subsection{A preparatory proposition}
\label{ss:prepprop}

The following proposition paves the way for the proofs of Theorems~\ref{thm:BK1} and
\ref{thm:BK2}.

\begin{prop} 
\label{prop:flow-proba}
Let $\Phi$ be a unit flow from $A$ to $B$, $\nu(C)= \Phi(C \times \Omega)$, $C\subset\Omega$ 
measurable, its left marginal, and $\ell(x,\dd y)$ a probability transition kernel such that 
$\Phi(\dd x, \dd y) = \nu(\dd x) \ell(x, \dd y)$. Let $Y=(Y_n)_{n \in \N_0}$ be the Markov chain 
with initial law $\P^\Phi(Y_0 \in C) = \nu(A \cap C)$, $C\subset\Omega$ measurable, and probability 
transition kernel $\ell$, and let $\tau_B^Y= \min\{n\in \N\mid Y_n \in B\}$. Then: 
\begin{enumerate} [(1)]
\item 
For all bounded non-negative measurable functions $f$, 
\be
\label{ineq:edge-proba}	
\int_{\Omega \times \Omega} f(x,y) \Phi(\dd x, \dd y) \geq
\E^\Phi\left[ \sum_{n=1}^{\tau_B^Y} f(Y_{n-1},Y_n) \right].
\ee
\item 
If $\Phi$ is absolutely continuous w.r.t.\ $K$,  then the Berman\tire Konsowa bound holds: 
\be
\label{ineq:cap}
\Cap(A,B) \geq \E^\Phi\left[ \left( \sum_{n=1}^{\tau_B^Y}  
\frac{\dd \Phi}{\dd K}(Y_{n-1},Y_n) \right)^{-1} \right]. 
\ee
\item 
If the flow is loop-free, then the paths $(Y_n)_{0 \leq n \leq \tau_B^Y}$ are self-avoiding 
$\P^\Phi$-a.s. 
\item 
If $\Phi(\Omega \times \Omega)<\infty$, then $\E^\Phi[ \tau_B^Y] <\infty$ and so $\tau_B^Y
<\infty$ $\P^\Phi$-a.s.    		
\end{enumerate}
\end{prop}

\begin{proof}
First we check that the measure $\P^\Phi$ does not depend on the precise choice of $\ell$. 
To this aim, we prove that for all $n\in \N$ the measure $\nu_n(C) = \P(Y_n \in C,\, \tau_B^Y 
\geq n+1)$ is absolutely continuous w.r.t.\ $\nu$. The proof is by induction on $n$. 

For $n=0$, the statement is true by the definition of $Y_0$. Suppose it is true for some 
$n\in \N_0$. Because $\Phi(\Omega \times A) =0$, $Y$ never returns to $A$ and we 
have $\nu_{n+1}(A) =0$. Hence we need only look at $\nu$-null sets $C\subset \Omega 
\backslash (A\cup B)$. Thus, let $C\subset \Omega \backslash (A\cup B)$ with $\nu(C)=0$. 
Kirchhoff's law yields $\Phi(\Omega \times C) = \Phi(C\times \Omega) =\nu(C) =0$, and
\be
\begin{aligned}
\nu_{n+1}(C) 
&= \P( Y_{n+1} \in C,\, \tau_B^Y \geq n+2) \\
&= \P( Y_{n+1} \in C,\, \tau_B^Y \geq n+1)\\ 
&= \int_{(\Omega \backslash B) \times C}  \nu_n(\dd x) \ell(x,\dd y)\\
&= \int_{(\Omega \backslash B) \times C} \frac{\dd \nu_n}{\dd \nu}(x) \Phi(\dd x,\dd y) \\
&= \int_{\Omega \backslash B} \frac{\dd \nu_n}{\dd \nu}(x) \Phi(\dd x,C) =0.
\end{aligned}
\ee
It follows that $\nu_{n+1}$ is absolutely continuous w.r.t.\ $\nu$. To conclude, we note that 
the equation $\Phi(\dd x,\dd y) = \nu(\dd x) \ell(x,\dd y)$ determines $\ell(x,C)$ up to changes 
for $x$ in $\nu$-null sets. Since $Y$ does not see $\nu$-null sets except possibly in $B$, the 
law of $Y$  stopped upon reaching $B$ is unaffected by this ambiguity. 

\medskip\noindent
(1) Let $f\colon\,\Omega \times \Omega \to [0,\infty)$ be bounded and measurable. We 
prove by induction on $n$ that
\be
\label{eq:flow-path1}
\int_{\Omega \times \Omega} f(x,y) \Phi(\dd x,\dd y) 
= \E \left[ \sum_{k=1}^{n \wedge \tau_B} f(Y_{k-1},Y_k) \right] + R_n
\ee 
with remainder term
\be
\label{eq:flow-path2}
R_n  = \int_{\Omega \times  (A \cup B)^\mathrm{c}} \Phi( \dd x_0,\dd x_1) F_n(x_1),
\ee
where
\be
F_n(x_1) = \int_{[(A \cup B)^\mathrm{c}]^{n} \times \Omega}
\ell(x_1,\dd x_2) \times \cdots \times \ell(x_n,\dd x_{n+1}) f(x_n,x_{n+1}).
\ee
For $n=0$ the claim is obvious. Suppose that (\ref{eq:flow-path1})--(\ref{eq:flow-path2}) 
hold for some $n\in \N_0$. Then 
\be
\begin{aligned}
R_n 
& = \int_{\Omega \times (A \cup B)^\mathrm{c} } \nu(\dd x_0) \ell(x_0,\dd x_1) F_n(x_1) \\
& = \int_{(A\cup B) \times  (A \cup B)^\mathrm{c}} \nu(\dd x_0) \ell(x_0,\dd x_1) F_n(x_1) \\
& \qquad + \int_{\Omega \times  [(A \cup B)^\mathrm{c}]^2} 
\Phi(\dd x_{-1},\dd x_0) \ell(x_0,\dd x_1) F_n(x_1) \\
& = \int_{(A\cup B) \times (A \cup B)^\mathrm{c}} \nu(\dd x_0) \ell(x_0,\dd x_1) F_n(x_1)+ R_{n+1},
\end{aligned}	
\ee
where in the second equality we use Kirchhoff's law to rewrite $\nu(\dd x_0)$ as the right marginal 
of $\Phi$. Hence we obtain 
\be  
R_n = \E\Bigl[ f(Y_{n},Y_{n+1}) \mathbf{1}_{\{\tau_B^Y \geq n+1\}} \Bigr] + R_{n+1},
\ee 
where we use that $\Phi(B\times \Omega)=0$. The inequality in \eqref{ineq:edge-proba} follows 
by estimating $R_n \geq 0$ and letting $n\to \infty$ in \eqref{eq:flow-path1}. 

\medskip\noindent
(2) Recall \eqref{eq:dirichlet-form}. Let $\phi=\dd\Phi/\dd K$ and estimate, for any 
$h \in \mathcal{V}_{AB}$, 
\be
\begin{aligned}
\mathcal{E}(h) &\geq \frac12 \int_{\Omega \times \Omega} \bigl[h(y)-h(x)\bigr]^2 
\bigl( \mathbf{1}_{\{\phi(x,y)>0\}} + \mathbf{1}_{\{\phi(y,x) >0\}}\bigr) K(\dd x,\dd y)\\
& = \int_{\Omega \times \Omega} \bigl[h(y)-h(x)\bigr]^2 
\mathbf{1}_{\{\phi(x,y)>0\}}  K(\dd x,\dd y) \\
&=  \int_{\Omega \times \Omega} \frac{\bigl[h(y)-h(x)\bigr]^2}{\phi(x,y)}
\mathbf{1}_{\{\phi(x,y)>0\}} \Phi(\dd x,\dd y)\\
&\geq \E^\Phi\left[\sum_{n=1}^{\tau_B^Y}  
\frac{\bigl[h(Y_n)-h(Y_{n-1})\bigr]^2}{\phi(Y_{n-1},Y_n)} 
\mathbf{1}_{\{\phi(Y_{n-1},Y_n)>0\}} \right],
\end{aligned}
\ee
where the first two lines use condition (4) in Definition~\ref{def:flow} and the symmetry of $K$, and the second inequality uses \eqref{ineq:edge-proba}. Take the infimum over
$h \in \mathcal{V}_{AB}$ and use \eqref{capdef}, to obtain
\be
\Cap(A,B) \geq \E^\Phi\left[ \inf_{h \in \mathcal{V}_{AB}}  \sum_{n=1}^{\tau_B^Y}  
\frac{\bigl[h(Y_n)-h(Y_{n-1})\bigr]^2}{\phi(Y_{n-1},Y_n)} 
\mathbf{1}_{\{\phi(Y_{n-1},Y_n)>0\}}\right].
\ee  
The infimum under the expectation can be easily computed, and equals
\be
\left( \sum_{n=1}^{\tau_B^Y}  \phi(Y_{n-1},Y_n) \right)^{-1}
\ee
because $h(Y_0)=1$ and $h(Y_{\tau_B^Y}) = 0$ (see \eqref{Ehest2} below). Hence
\eqref{ineq:cap} holds.
 
\medskip\noindent
(3) Let $\chi \subset \Omega \times \Omega$ be a loop-free measurable set with 
$\Phi(\chi^\mathrm{c})=0$. By construction, the path $(Y_0,\ldots,Y_n)$ with $n \leq 
\tau_B^Y$ has only transitions $(Y_{j-1},Y_j) \in \chi$ a.s. Since $\chi$ is loop-free, 
$Y$ stopped upon reaching $B$ is self-avoiding. 

\medskip\noindent
(4) Apply the inequality in \eqref{ineq:edge-proba} to the constant function $f \equiv 1$. 
This gives 
\be 
\E^\Phi\left[ \tau_B^Y \right] = \E^\Phi\left[ \sum_{n\in\N} 
\mathbf{1}_{\{\tau_B^Y \geq n \} } \right] \leq \Phi(\Omega \times \Omega). 
 \ee
If $\Phi$ is finite, then $\tau_B^Y$ has finite expectation and hence is finite $\P^\Phi$-a.s.
\end{proof}


\subsection{A discrepancy lemma}
\label{ss:discrlem}

To get equalities in (\ref{ineq:edge-proba})--(\ref{ineq:cap}) we need an extra argument,
which is based on the following lemma.

\begin{lemma}
\label{lem:psi}
Let $\Phi$ be a unit flow from $A$ to $B$ that is finite, i.e., $\Phi(\Omega \times \Omega) 
< \infty$, and let $Y=(Y_n)_{n \in \N_0}$ be the associated Markov chain as in 
Proposition~{\rm \ref{prop:flow-proba}}. Let $\tilde \Phi$ and $\Psi$ be the measures on 
$\Omega \times \Omega$ defined by 
\be
\label{eq:tphi}
\begin{aligned}
\int_{\Omega\times \Omega} f(x,y) \tilde \Phi(\dd x,\dd y) 
& = \E^\Phi\left[ \sum_{n\in\N} f(Y_{n-1},Y_n) \mathbf{1}_{\{\tau_B^Y \geq n\}} \right], \\
\Psi & = \Phi - \tilde \Phi. 
\end{aligned}
\ee
Then $\tilde \Phi$ is a unit flow from $A$ to $B$, and $\Psi$ satisfies Kirchhoff's law, i.e., 
$\Psi(C \times \Omega) = \Psi(\Omega \times C)$ for all $C\subset \Omega$. 
\end{lemma}

\begin{proof}
Proposition~\ref{prop:flow-proba} tells us that $\tilde \Phi \leq \Phi$, and so $\Psi$ is a 
non-negative measure on $\Omega \times \Omega$, satisfying $\Psi \leq \Phi$. Both 
$\tilde \Phi$ and $\Psi$ inherit Kirchhoff's law on $(A\cup B)^\mathrm{c}$ from $\Phi$. 
They also inherit the properties $\Phi(\Omega \times A)=0$ and $\Phi(B \times \Omega)
=0$. Furthermore,
\be
\tilde \Phi(A \times \Omega) = \P^\Phi\bigl( (Y_0,Y_1) \in A \times  \Omega \bigr) =1, 
\ee
where we use that $Y$ does not return to $A$ after time $0$. Thus, $n=1$ is the only 
summand contributing to \eqref{eq:tphi}. Similarly, 
\be
\tilde \Phi(\Omega \times B) =  \P^\Phi( \tau_B^Y<\infty ) =1,
\ee
where we use that the unit flow is finite and that by Proposition~\ref{prop:flow-proba} the 
hitting time $\tau_B^Y$ is $\P^\Phi$-a.s.\ finite. Therefore $\tilde \Phi$ is a unit flow 
from $A$ to $B$. Moreover,
\be 
\Psi(A \times \Omega) = \Phi(A \times \Omega)- \tilde\Phi(A \times \Omega) =1-1 = 0 
\ee
and $\Psi(\Omega \times B) =0$. It follows that $\Psi(C\times \Omega) = \Psi(\Omega 
\times C)$ for all measurable $C\subset\Omega$. 
\end{proof}


\subsection{The harmonic flow}
\label{ss:harmflow}

\begin{lemma} 
\label{lem:hf}
$\Phi_{AB}$ is a finite loop-free unit flow from $A$ to $B$. 
\end{lemma}

\begin{proof} 
We check that properties (1)--(4) in Definition~\ref{def:flow} hold for $\Phi=\Phi_{AB}$. 

\medskip\noindent
(1) By \eqref{eq:harmonic-flow}, we have 
\be 
\Phi_{AB}(\Omega \times A)  = \frac{1}{\Cap(A,B)} \int_{\Omega \times A} 
\bigl[ h_{AB} (x) - h_{AB}(y)\bigr]_+ K(\dd x,\dd y) = 0
\ee
because $h_{AB}(x) \leq 1$ for $x \in \Omega$ and $h_{AB}(y)=1$ for $y \in A$.  A similar 
argument shows that $\Phi_{AB}(B\times \Omega) = 0$. 

\medskip\noindent
(2) By (\ref{eq:KPk})--(\ref{eq:generator}), \eqref{eq:cap-def1} and \eqref{eq:harmonic-flow}, 
we have
\be
\begin{aligned}
\Phi_{AB}(A\times \Omega) 
&= \frac{1}{\Cap(A,B)} \int_{A} \mu(\dd x) \int_\Omega \bigl[h_{AB}(x)- h_{AB}(y)] k(x,\dd y) \\
& = \frac{1}{\Cap(A,B)} \int_{A} \mu(\dd x) \bigl(-L h_{AB}\bigr) (x) = 1
\end{aligned}
\ee
because $h_{AB}(x)=1$ for $x \in A$ and $h_{AB}(y) \leq 1$ for $y \in \Omega$. The fact that
$\Phi_{AB}(\Omega \times B) = 1$ follows from the symmetry relations $h_{BA} = 1- h_{AB}$,  
$\Phi_{AB}(C\times D) = \Phi_{BA} (D \times C)$ and $\Cap(A,B) = \Cap(B,A)$. 

\medskip\noindent
(3) Let $C\subset \Omega \backslash (A\cup B)$. Because of the symmetry of $K$, we have 
\be
\begin{aligned}
\Phi(\Omega \times C) 
& = \int_{\Omega \times C} \bigl[ h_{AB}(x) - h_{AB}(y)\bigr]_+ K(\dd y,\dd x) \\
& = \int_{C\times \Omega} \bigl[ h_{AB}(x) - h_{AB}(y) \bigr]_- K(\dd x,\dd y).
\end{aligned}
\ee
Therefore, by  (\ref{eq:KPk})--(\ref{eq:generator}) and the fact that $[\cdot]_+ - [\cdot]_- = [\cdot]$, 
\be
\begin{aligned}
\Phi(C \times \Omega) - \Phi(\Omega \times C) 
&= \int_{C\times \Omega}  \bigl[h_{AB}(x) - h_{AB}(y)\bigr] K(\dd x,\dd y)\\
&= \int_C \mu(\dd x) (-L h_{AB})(x) = 0, 
\end{aligned}
\ee
because $\mu$ is invariant. Thus, Kirchhoff's law holds. 

\medskip\noindent
(4) Define 
\be
\chi = \bigl\{ (x,y) \in \Omega \times \Omega \mid h_{AB}(x) > h_{AB}(y) \bigr\}.  
\ee
Clearly, $\Phi_{AB}(\chi^\mathrm{c})=0$ and $\chi$ contains no loops. Hence $\Phi_{AB}$ 
is a loop-free unit flow from $A$ to $B$. The flow is finite because 
\be 
\Phi_{AB}(\Omega \times \Omega) 
= \frac{1}{\Cap(A,B)} \int_{\Omega \times \Omega} 
\bigl[ h_{AB}(x) - h_{AB}(y) \bigr]_+ K(\dd x,\dd y) 
\leq \frac{K(\Omega\times \Omega)}{\Cap(A,B)}  < \infty,
\ee
where we use that $h_{AB} \leq 1$ and $K(\Omega \times \Omega) <\infty$. 
\end{proof}

\begin{lemma}
\label{lem:harm}
$Y$ with law $\P^{\Phi_{AB}}$ satisfies $\tau^Y_B<\infty$ a.s., its paths are 
self-avoiding, and equality holds in \eqref{ineq:cap}. 
\end{lemma}

\begin{proof}
We proceed as in Lemma~\ref{lem:psi} and exploit the fact that $(h_{AB} (Y_n))_{n\in\N_0}$ 
is $\P^{\Phi_{AB}}$-a.s.\ strictly decreasing. Define $\Psi_{AB}$ as in Lemma~\ref{lem:psi} for $\Phi
= \Phi_{AB}$, i.e., 
\be 
\Psi_{AB}(f) = \int_{\Omega \times \Omega} f \dd \Psi_{AB} 
= \int_{\Omega \times \Omega} f \dd \Phi_{AB} 
- \E^{\Phi_{AB}} \left[ \sum_{n=1}^{\tau_B^Y} f(Y_{n-1},Y_n) \right].
\ee
We have to show that $\Psi_{AB} =0$. We already know that $\Psi_{AB}$ is a finite measure 
satisfying Kirchhoff's law in all of $\Omega$. Suppose that $\Psi_{AB}(\Omega \times \Omega)>0$. 
Let $\hat \nu(\dd x)$ be its marginal and $\hat \ell(x, \dd y)$ any probability transition kernel such 
that $\hat \nu(\dd x) \hat \ell(x,\dd y) = \Psi_{AB}(\dd x,\dd y)$. Without loss of generality we may 
assume that $\hat \nu(\Omega) =1$. Let $\hat Y = (\hat Y_n)_{n\in \N}$ be a stationary Markov 
process with probability transition kernel $\hat \ell$ and initial distribution $\P( \hat Y_0 \in C) 
= \hat \nu(C)$, $C \subset \Omega$ measurable.  A Poincar\'e-recurrence-type argument shows 
that $\hat Y$ returns to $C$ infinitely often for every $C$ with $\hat \nu(C) >0$. On the other hand, 
we know that the set $\{ (x,y) \in \Omega \times \Omega \mid h_{AB}(x) \leq h_{AB}(y) \}$ has 
$\Phi_{AB}$-measure $0$. Since $\Psi_{AB}(f) \leq \Phi_{AB}(f)$, this set also has $\Psi_{AB}$-measure 
$0$. It therefore follows that there is an $m>0$ such that 
\be
\Psi_{AB}\Bigl( \big\{ (x,y) \in \Omega \times \Omega \mid h_{AB}(x) > m > h_{AB}(y) \big\} \Bigr) >0.
\ee

Put $C = \{x\in \Omega \mid h_{AB}(x) >m\}$. Then $\hat\nu(C) = 1$ and $\Psi_{AB}(C \times C)
=0$. Therefore once $\hat Y$ has left $C$ it cannot come back to $C$, contradicting the fact 
that it returns to $C$ infinitely often. Thus, the assumption $\Psi_{AB}(\Omega \times \Omega)>0$ 
leads to a contradiction. We conclude that $\Psi_{AB}=0$, and so there is equality in 
\eqref{ineq:edge-proba}. 

To show that equality holds in \eqref{ineq:cap}, i.e., $\Phi_{AB}$ is a maximizer of \eqref{eq:Cap2}, 
we compute the right-hand side of \eqref{ineq:cap}. Note that $(\dd \Phi_{AB}/\dd K) (x,y) 
= [h_{AB} (x) - h_{AB}(y)]_+$, and recall that $\P^{\Phi_{AB}}$-a.s.\ $(h_{AB}(Y_n))_{n\in\N_0}$ is 
strictly increasing (until it reaches $B$). We have 
\be
\label{capharm}
\begin{aligned}
&\frac{1}{\Cap(A,B)}\, \E^{\Phi_{AB}} \left[ \left( \sum_{n=1}^{\tau_B^Y} 
\frac{\dd \Phi_{AB}}{\dd K}(Y_{n-1},Y_n) \right)^{-1} \right] \\
&\qquad = \E^{\Phi_{AB}} \left[ \left(\sum_{n=1}^{\tau_B^Y} 
\bigl[ h_{AB}(Y_{n-1}) - h_{AB}(Y_n) \bigr] \right)^{-1} \right] \\
&\qquad = \E^{\Phi_{AB}} \left[ \Bigl(h_{AB}(Y_0) - h_{AB}\bigl(Y_{\tau_B^Y}\bigr) \Bigr)^{-1} \right]
= \E^{\Phi_{AB}} \Bigl[ (1-0)^{-1} \Bigr] = 1, 
\end{aligned}
\ee
which is the desired result. 
\end{proof}


\subsection{Proof of Theorems \ref{thm:BK1} and \ref{thm:BK2}}
\label{ss:prooftheorems}

\begin{proof}
Theorem~\ref{thm:BK2} follows from Proposition~\ref{prop:flow-proba}(2--4) and 
Lemma~\ref{lem:harm}. To prove Theorem~\ref{thm:BK1} we argue as follows.

Let $\P \in \mathcal{P}_{AB}^K$ and put $\phi= \dd \Phi_\P/ \dd K$. Pick 
$h \in \mathcal{D}(\mathcal{E})$ such that $h|_A=1$ and $h|_B =0$. Then,
by \eqref{eq:dirichlet-form} and \eqref{eq:edge-proba},
\be
\label{Ehest1}
\begin{aligned} 
\mathcal{E}(h) 
&\geq \tfrac12 \int_{\Omega \times \Omega} \bigl[ h(y) - h(x) \bigr]^2 
[\mathbf{1}_{\{\phi(x,y) >0\}} + \mathbf{1}_{\{\phi(y,x) >0\}}]\, K(\dd x,\dd y) \\
&= \int_{\Omega \times \Omega} \bigl[ h(y) - h(x) \bigr]^2 
\mathbf{1}_{\{\phi(x,y) >0\}}\, K(\dd x,\dd y) \\
&=  \int_{\Omega \times \Omega} 
\frac{ \bigl[ h(y) - h(x) \bigr]^2}{\phi(x,y)} \mathbf{1}_{\{\phi(x,y) >0\}} \Phi_\P (\dd x,\dd y)\\
&= \E \left[ \sum_{(x,y) \in \gamma} \frac{ \bigl[ h(y) - h(x) \bigr]^2}{\phi(x,y)} 
\mathbf{1}_{\{\phi(x,y) >0\}} \right]\\
&= \E \left[ \sum_{(x,y) \in \gamma} \frac{ \bigl[ h(y) - h(x) \bigr]^2}{\phi(x,y)} \right],
\end{aligned}
\ee
where in the last line we use that $\phi(x,y)>0$ for all $(x,y) \in \gamma$ for $\P$-a.a.\ paths 
$\gamma$. The solution to the one-dimensional harmonic problem is trivial, namely, for every 
 $\gamma = (\gamma_0,\ldots,\gamma_\tau) \in \Gamma_{AB}$ we have
\be
\label{Ehest2}
\begin{aligned} 
&\inf\left\{  \sum_{(x,y) \in \gamma} \frac{ \bigl[ h(y) - h(x) \bigr]^2}{\phi(x,y)}
~\Big|~ h\colon\gamma \to [0,1],\ h(\gamma_0)=1,\ h(\gamma_\tau)=0 \right\}\\
&\qquad \geq \left( \sum_{(x,y) \in \gamma} \phi(x,y) \right)^{-1},
\end{aligned}
\ee 
where the inequality is an identity when $\gamma$ is self-avoiding. Combining 
(\ref{Ehest1})--(\ref{Ehest2}) and recalling \eqref{capdef}, we get 
\be
\label{capbd}
\Cap(A,B) \geq \E\left[\left( \sum_{(x,y) \in \gamma} 
\frac{\dd\Phi_\P}{\dd K}(x,y) \right)^{-1}\right].
\ee

Next, recall that $\P^{\Phi_{AB}}$ is the probability measure on $\Gamma_{AB}$ associated 
with the harmonic flow $\Phi_{AB}$. We know that $\P^{\Phi_{AB}}$ is a probability measure 
on finite self-avoiding paths from $A$ to $B$, and so we have
\be
\begin{aligned}
\E^{\Phi_{AB}}\left[ \left( \sum_{(x,y) \in \gamma} \frac{\dd \Phi_{AB}}{\dd K}(x,y) \right)^{-1} \right] 
&=  \E^{\Phi_{AB}} \left[ \left( \sum_{n=1}^{\tau_B^Y} \frac{\dd \Phi_{AB}}{\dd K}
(Y_{n-1},Y_n) \right)^{-1} \right]\\ 
&= \Cap(A,B),
\end{aligned}
\ee
where the first equality uses the definition of $Y$ in Proposition~\ref{prop:flow-proba} and the 
second equality uses \eqref{capharm}. Thus, equality is achieved in \eqref{capbd} for 
$\P=\P^{\Phi_{AB}}$. This completes the proof of Theorem~\ref{thm:BK1}.
\end{proof}


\subsection{Proof of Theorem~\ref{thm:BK12ext}}
\label{ss:BK12extension}

The extension of Theorems~\ref{thm:BK1} and \ref{thm:BK2} from positive recurrent
$X$ to null-recurrent and transient $X$ proceeds via a \emph{truncation argument}. 
In Appendix~\ref{ss:truncation} we let $(\Omega_n)_{n\in \N}$ be an increasing 
sequence of measurable subsets of $\Omega$ with $\cup_{n\in\N} \Omega_n = \Omega$ 
such that $A \cup B \subset \Omega_n$ and $K(\Omega_n \times \Omega_n)<\infty$ for 
all $n\in\N$. We show that if $\Cap_n(A,B)$ denotes the capacity for the reversible Markov 
jump process $X^n$ obtained from $X$ by \emph{suppressing} jumps outside $\Omega_n$, 
then
\be
\label{eq:captruncapprox}
\lim_{n\to\infty} \Cap_n(A,B) = \Cap(A,B) 
\ee  
when $K((A \cup B) \times \Omega)<\infty$. 

In order to prove the extension of Theorem~\ref{thm:BK2}, we argue as follows. Let $h_{AB}^n$ 
be the harmonic function on $\Omega_n$, and let $\Phi_{AB}^n$ be the harmonic flow on 
$\Omega_n$ given by \eqref{eq:harmonic-flow}: 
\be
\Phi_{AB}^n(\dd x,\dd y) = \frac{1}{\Cap_n(A,B)}\,\bigl[h_{AB}^n(x)-h_{AB}^n\bigr]_+
\mathbf{1}_{\Omega_n}(x)\mathrm{1}_{\Omega_n}(y)\,K(\dd x,\dd y).
\ee
Note that, for all $n\in\N$, $\Phi_{AB}^n \ll K$ and $\Phi_{AB}^n \in \mathcal{U}_{AB}^K$, where
the latter is the set of finite unit flows on $\Omega$. Therefore, by Theorem~\ref{thm:BK2}, we have
\be
\label{eq:captruncharm}
\Cap_n(A,B) = \E^{\Phi_{AB}^n}\left[ \left( \sum_{n=1}^{\tau_B^Y} 
\frac{\dd \Phi_{AB}^n}{\dd K}(Y_{n-1},Y_n) \right)^{-1} \right],
\ee
where we use that, for all $n\in\N$, $X^n$ has finite total conductance (i.e., is positive recurrent). 
Combining \eqref{eq:captruncapprox} and \eqref{eq:captruncharm}, we obtain
\be \label{eq:ctc}
\Cap(A,B)  \leq \sup_{\Phi \in \mathcal{U}_{AB}^K} 
\E^\Phi\left[ \left( \sum_{n=1}^{\tau_B^Y} \frac{\dd \Phi}{\dd K}(Y_{n-1},Y_n) \right)^{-1} \right].
\ee 
But the reverse inequality was already proved in Proposition~\ref{prop:flow-proba}(2).

In order to prove the extension of Theorem~\ref{thm:BK1}, we argue as follows. Equation \eqref{eq:Cap1} 
with $\geq$ instead of $=$ is a consequence of Proposition~\ref{prop:flow-proba}(2). For the reverse 
inequality, note that in \eqref{eq:captruncharm} we have $\Phi_{AB}^n= \Phi_{\P^{n}}$ with $\P^n
= \P^{\Phi_{AB}^n} \in \mathcal{P}_{AB}^K$. Passing to the limit $n\to\infty$, we obtain \eqref{eq:ctc} with $\Phi$ replaced by $\Phi_\P$ and the supremum over $\Phi\in \mathcal{U}_{AB}^K$  replaced by the supremum over $\P\in\mathcal{P}_{AB}^K$.


\appendix

\section{Appendix~A. Potential-theoretic ingredients}
\label{app1}

In Section~\ref{app1:jump-process} we provide the details of the construction of the Markov 
jump process $X=(X_t)_{t \geq 0}$ introduced in Section~\ref{ss:setting}. In Section~\ref{app1:Diri} 
we check that the Dirichlet principle in \eqref{capdef} has a unique solution. In Sections~\ref{app1:DT}
and \ref{ss:TPproved} we give the proof of the Dirichlet principle and the Thomson principle. In 
Section~\ref{ss:truncation} we show that the capacity can be obtained as the limit of certain 
truncated capacities. This is crucial for the extension from finite total conductances ($K(\Omega
\times\Omega)<\infty$, corresponding to $X$ with positive recurrent $Z$) to infinite total conductances 
($K(\Omega \times \Omega)=\infty$).
 
 
\subsection{Jump process} 
\label{app1:jump-process}

Recall that a \emph{kernel} on $(\Omega,\mathcal{F})$ (with $\Omega$ the state space and 
$\mathcal{F}$ the Borel $\sigma$-algebra) is a map $k\colon \Omega \times \mathcal{F} \to 
[0,\infty)$ such that $x\mapsto k(x,B)$ is measurable for every $B\in \mathcal{F}$, and $k(x,\cdot)$ 
is a measure for every $x\in \Omega$. We assume that $k(x,\Omega)<\infty$ for all $x\in \Omega$, 
and that $k(x,\dd y)$ admits a $\sigma$-finite reversible measure $\mu$, i.e., $\mu(\dd x)k(x,\dd y) 
= \mu(\dd y) k(x,\dd y)=K(\dd x,\dd y)$. Thus, $K$ is the unique measure on the product space 
$\Omega \times \Omega$ such that 
\be
K(C\times D) = \int_C \mu(\dd x) k(x,D), \qquad C,D \subset \Omega \text{ measurable}.
\ee
Set 
 \be 
\lambda(x) = k(x,\Omega)
\ee
and assume that $\lambda(x)>0$ for $\mu$-a.a.\ $x \in \Omega$. The \emph{minimal jump process} 
$X=(X_t)_{t\geq 0}$ associated with $k(x,\dd y)$ is defined as follows. Let $(Q_t(x,\dd y))_{t\geq 0}$ 
be the minimal solution of the backward Kolmogorov equation (existence and uniqueness are proven 
in Feller~\cite[Chapter 3]{feller-vol2}). Thus,  $Q_0(x,\dd y) = \delta_x(\dd y)$ is the identity kernel and, 
for all $t\geq 0$, $x\in \Omega$ and $C\subset \Omega$ measurable, 
\be 
\label{eq:kolmogorov}
\frac{\partial Q_t}{\partial t} (x,C) = - \lambda(x) Q_t(x,C) + \int_\Omega k(x,\dd y) Q_t(y,C).  
\ee
The right-hand side is written as $(L Q_t(\cdot,C))(x)$ with $L$ the generator of $X$.  The minimal 
solution satisfies $Q_t(x,\Omega) \leq 1$ for all $x\in \Omega$, $t\geq 0$ and therefore defines 
a Markov process $(X_t)_{t \geq 0}$ with a possibly finite lifetime $\zeta>0$.

Next, we specify the domains of the Dirichlet form and the generator, and check that $\mu$ is 
a reversible measure. Let $L_0$ be the operator in $L^2(\Omega,\mu)$ with domain 
\be
\label{eq:dlzero} 
\mathcal{D}(L_0) = \Bigl \{ f\in L^2(\Omega,\mu) \mid \exists\,n \in \N\colon\, 
\mu\bigl( \{x\in \Omega\mid f(x) \neq 0,\, \lambda(x) >n \}\bigr) = 0  \Bigr\}
\ee 
and $L_0 f = L f$ as in \eqref{eq:generator}.  Let $\mathcal{E}^*$ be the quadratic form with 
domain 
\be 
\mathcal{D}(\mathcal{E}^*)  = \Bigl \{ f \in L^2(\Omega,\mu) \mid 
\int_{\Omega\times\Omega} \bigl[ f(y) - f(x) \bigr]^2 K(\dd x,\dd y) <\infty \Bigr \}
\ee
and $\mathcal{E}^*(f) = \mathcal{E}(f)$ as in \eqref{eq:dirichlet-form}. Set  
\be 
\label{eq:norm-one}  
\|f\|_1 = \Bigl( \int_\Omega f^2 \dd \mu + \mathcal{E}(f) \Bigr)^{1/2}.  
\ee 
Note that $\mathcal{D}(L_0) \subset \mathcal{D}(\mathcal{E})$. Indeed, if $f\in L^2(\Omega,\mu)$ is supported 
in $\{x \in \Omega \mid \lambda(x) \leq n\}$, then 
\be
\int_{\Omega \times \Omega} \bigl[ f(y) - f(x) \bigr]^2 K(\dd x,\dd y) 
\leq  4 \int_{\Omega} f(x)^2 \lambda(x) \mu(\dd x) \leq 4 n \int_\Omega f(x)^2\mu(\dd x)<\infty
\ee
by the inequality $(a-b)^2\leq 2 (a^2 + b^2)$ and the symmetry of $K$. Let $\mathcal{D}(\mathcal{E})
\subset \mathcal{D}(\mathcal{E}^*)$ be the closure of $\mathcal{D}(L_0)$ with respect to $\|\cdot\|_1$. 

For explosive processes the operator $L_0$ can have several self-adjoint extensions. The next 
proposition says that the generator $L$ of the minimal jump process $X$ is the \emph{Friedrichs extension} 
of $L_0$. For discrete state spaces, this result was shown by Silverstein~\cite{silverstein1}, \cite{silverstein2}. 

\begin{prop}{\rm (Chen~\cite[Theorem 3.6]{chen})} 
\label{prop:dirichlet}
The following hold:\\ 
{\rm (1)} $\mathcal{E}$ with domain $\mathcal{D}(\mathcal{E})$ is a  closed quadratic form in 
$L^2(\Omega,\mu)$.\\  
{\rm (2)} There is a  unique self-adjoint operator $L$ with domain $\mathcal{D}(L) \subset \mathcal{D}
(\mathcal{E})$ such that $\mathcal{E}(f) = \int_\Omega f (- L f) \dd \mu$ for all $f\in \mathcal{D}(L)$. 
The operator $L$ is an extension of $L_0$.\\ 		
{\rm (3)} Let $(Q_t)_{t \geq 0}$ be the minimal solution of the backward Kolmogorov equation in 
\eqref{eq:kolmogorov}. For all $t\geq 0$ and all $f\in L^2(\Omega,\mu)$, 
\be
(Q_t f)(x) =  \bigl( e^{t L} f\bigr)(x) \quad \mu\text{-a.e.}
\ee
\end{prop}

\noindent
Part (3) shows that $(Q_t)_{t\geq 0}$ is  self-adjoint in $L^2(\Omega,\mu)$:

\begin{cor}
The measure $\mu$ is a reversible measure for the minimal jump process $X$.  
\end{cor} 

Finally, we cite the two conditions for non-explosion that we will need later. 

\begin{prop}{\rm (Chen~\cite[Corollary 3.7]{chen})}
\label{prop:non-explosion1}
The following are equivalent:\\ 
{\rm (1)} $\mathcal{D}(\mathcal{E}) =  \mathcal{D}(\mathcal{E}^*)$, i.e., $\mathcal{D}(L_0)$ is 
dense in $\mathcal{D}(\mathcal{E}^*)$ with respect to $\|\cdot\|_1$.\\
{\rm (2)} $L_0$ has no self-adjoint extensions other than the operator $L$ defined in 
Proposition~{\rm \ref{prop:dirichlet}}.\\ 
{\rm (3)} $X$ is non-explosive: $\P_x(\zeta <\infty) = 0$ for $\mu$-a.a.\ $x\in \Omega$, with 
$\zeta$ denoting the lifetime.  
\end{prop} 

\begin{prop}	{\rm (Chen~\cite[Corollary 3.8]{chen})} 
\label{prop:non-explosion2}
If $K(\Omega \times \Omega) = \int_{\Omega} \lambda(x) \mu(\dd x) <\infty$, then $X$ is 
non-explosive in the sense of Proposition~{\rm \ref{prop:non-explosion1}}. 
\end{prop}

The non-explosion criterion of Proposition~\ref{prop:non-explosion2} amounts to \emph{positive 
recurrence} of the underlying jump chain, which will appear in the analysis of the Dirichlet problem 
given below. Indeed, let $Z=(Z_n)_{n\in\N_0}$ be the \emph{jump chain} associated with $X$, 
i.e., $Z_0=X_0$ and $Z_n=X_{t_n+}$ with $t_n+$ the time right after the (random) time $t_n$ 
of the $n$-th jump of $X$. If $\lambda(y)=0$ for some $y$ then it is possible that $X$ makes 
only finitely a finite number $N$ of jumps, in which case we set $Z_m= Z_N$, $m \geq N$. 
With this convention $Z$ is a Markov chain with probability transition kernel
\be
p(x,\dd y) = \begin{cases}
\lambda(x)^{-1} k(x,\dd y),
&\quad \lambda(x)>0,\\
\delta_x(\dd y),
& \quad\lambda(x)=0, 
\end{cases}  
\ee
We have $k(x,\dd y) = \lambda(x) p(x,\dd y)$, and so $\nu(\dd x) = \lambda(x) \mu(\dd x)$ is a 
reversible measure for the jump chain $Z$. The condition $K(\Omega \times\Omega)<\infty$ 
holds if and only $\nu$ has finite mass, which for irreducible Markov chains on discrete state 
spaces is equivalent to positive recurrence (Stroock~\cite{stroock}). More generally, if 
$K(\Omega\times\Omega)<\infty$, then a Poincar{\'e}-recurrence-type argument shows $Z$,
and hence also $X$, visits every set of positive measure infinitely often. 


\subsection{Dirichlet problem} 
\label{app1:Diri}
 
We say that $f\colon\,\Omega \to \R$ solves the Dirichlet problem when 
\be
\label{Diriproblem}
f =1 \text{ on } A , \quad f = 0 \text{ on } B, \quad - Lf = 0 \text{ on } \Omega \backslash (A\cup B).
\ee 
We recall the definition of the hitting time $\tau_A = \inf\{t>0 \mid X_t \in A\}$. Let $Z=(Z_n)_{n\in \N_0}$ 
(with the convention $\inf \emptyset = \infty$) be the jump chain associated with $X=(X_t)_{t\geq 0}$, 
and set $\sigma_C=\inf \{n\in \N_0 \mid Z_n \in C\}$ for $C\subset \Omega$. 
 
\begin{prop} 
\label{prop:dirichlet-problem} 
Let $A,B\subset \Omega$ be disjoint measurable sets. Then 
\be 
g_{AB}(x) = \P_x(\tau_A<\zeta,\,\tau_A<\tau_B)
\ee
is the minimal non-negative solution of the Dirichlet problem. If  $\P_x(\tau_{A\cup B} <\infty) = 1$ for 
$\mu$-a.a.\ $x\in \Omega$, then $g_{AB}$ is the unique bounded solution of the Dirichlet problem (up 
to $\mu$-null sets). 
\end{prop} 

\begin{proof}
First, note that $g_{AB}(x) = 1$ for all $x\in A$ and $g_{AB}(x)=0$ for $x\in B$. This is because the jump 
rates $\lambda(x)$ are finite, so that $\P_x$-a.s.\ there is an $\eps>0$ such that $X_t =x$ for all $t \in 
[0,\eps)$. Next, note that $g_{AB}$ can be written in terms of the jump chain as  
\be
g_{AB}(x) = \P_x( \sigma_A <\infty,\, \sigma_A<\sigma_B).  
\ee
The reader may check that $\{\tau_A<\zeta \} = \{\sigma_A < \infty\}$: if $X$ explodes before hitting $A$ 
or has infinite life-time and never hits $A$, then $Z$ makes infinitely many jumps without hitting $A$, 
and vice-versa. For $x\in \Omega \backslash (A\cup B)$,
\be 
g_{AB}(x) = \P_x( Z_1 \in A) + \int_{\Omega\backslash A} p(x, \dd y)\, 
\P_y(\sigma_A <\infty,\, \sigma_A<\sigma_B ) = \bigl( p g_{AB}\bigr) (x).
\ee
Multiplying with $\lambda(x)$, we get 
\be
\begin{aligned}
0 &= \lambda(x) \bigl[g_{AB}(x)-(pg_{AB})(x)\bigr] 
= \int_\Omega \lambda(x) \bigl[g_{AB}(x)-g_{AB}(y)]\,p(x,\dd y) \\
&= \int_\Omega \bigl[g_{AB}(x)-g_{AB}(y)]\,k(x,\dd y)
= (-L g_{AB})(x) = 0, \qquad  x \in \Omega \backslash (A\cup B).
\end{aligned}
\ee

The proof that $g_{AB}$ is the minimal solution is analogous to the proof for discrete state spaces (see 
e.g.\ Norris~\cite[Chapter 4.2]{norris}). Indeed, let $h \geq 0$ be another solution of the Dirichlet problem 
and $x\in \Omega \backslash (A\cup B)$. Induction on $n$ shows that for all $n\in\N_0$, 
\be 
\label{eq:minimal-solution} 
h(x) = \sum_{k=0}^n \P_x(\sigma_A =k, \sigma_B\geq k+1) + r_n(x)
\ee
with
\be 
r_n(x) =  \int_{ [(A\cup B)^\mathrm{c}]^n \times \Omega} 
p(x,\dd y_1) p(y_1,\dd y_2) \cdots p(y_n,\dd y_{n+1}) h(y_{n+1}).
\ee
We estimate $r_n(x) \geq 0$, let $n\to \infty$, and obtain $h(x) \geq g_{AB}(x)$. Hence $g_{AB}$ is the 
minimal non-negative solution. 

Suppose that $\P_x(\sigma_{A\cup B} <\infty ) = \infty$, and let $h$ be a bounded solution of the 
Dirichlet problem. Then \eqref{eq:minimal-solution} holds and 
\be
|r_n(x)| \leq \|h(x)\|_\infty \P_x\bigl( \sigma_{A\cup B} \geq n+1\bigr),
\ee
which tends to zero as $n\to\infty$. It follows that $h(x) = g_{AB}(x)$. 
\end{proof}

Proposition~\ref{prop:dirichlet-problem} shows that if $K(\Omega\times \Omega)<\infty$ and $X$ is 
irreducible, then $g_{AB}$ is the unique bounded solution of the Dirichlet problem. For transient $X$, 
the solution is \emph{not} unique: for example, new solutions are obtained by adding multiples of the 
function $x\mapsto 1-\P_x( \tau_{A\cup B} < \zeta)$ (see also \eqref{eq:gab} below). 

For non-explosive recurrent $X$, $g_{AB}$ is exactly the function $h_{AB}$ defined in \eqref{eq:habdef}, 
and is equal to the unique minimizer in the Dirichlet principle. For transient $X$, $g_{AB}$ is no longer 
the minimizer (see Lemma~\ref{lem:harmonic-function} below).


\subsection{Dirichlet principle}
\label{app1:DT}

In this section we prove the Dirichlet principle. For $K(\Omega \times \Omega)<\infty$, i.e., positive 
recurrent $X$, the proof is simple and analogous to proofs for finite state spaces.  A key role is played 
by the Green identity 
\be \label{eq:green} 
\begin{aligned} 
&\tfrac12 \int_{\Omega \times \Omega} \bigl[ f(x) - f(y) \bigr] \bigl[ g(x) - g(y) \bigr] K(\dd x,\dd y)  \\
& \qquad =  \int_\Omega f(x) \bigl[ g(x) - g(y) \bigr] \mu(\dd x) k(x,\dd y) \\
& \qquad = \int_\Omega f(x) \bigl( - L g\bigr) (x) \mu(\dd x),  
\end{aligned} 
\ee
which holds for all bounded $f$ and $g$ when $K$ is finite. When $K$ is infinite, i.e., when $Z$ 
associated with $X$ is null-recurrent or transient, the integrals in the \eqref{eq:green} need not 
be absolutely convergent and the identity may fail. (A formal analogue is the identity $\int_{-\infty}^\infty 
f'(x) g'(x) \dd x = \int_{- \infty}^\infty f(x)[- g''(x)] \dd x$: the integration by parts works only if there are 
no boundary terms from infinity.) Using Cauchy\tire Schwarz, we see that the condition $\int_\Omega 
f(x)^2 \lambda(x) \mu(\dd x)<\infty$ is sufficient to ensure that the Green identity stays true. But in 
general  the minimizer of the Dirichlet principle need not satisfy this condition and therefore we need 
to proceed with caution.

Our strategy is as follows. First we show that the variational formula in \eqref{capdef} has a unique 
minimizer $h_{AB}$ (Lemma~\ref{lem:truncation} below). Next we check that the minimizer solves 
the Dirichlet problem (Lemma~\ref{lem:harmonic-function} below), and that the minimum is the limit 
of ``truncated'' minima (Lemma~\ref{lem:truncation} below). Finally we show that $\mathcal{E}(h_{AB}) 
= \int_A\bigl(-L h_{AB}\bigr) \dd \mu$ (Lemma~\ref{lem:charges} below). 

Recall the Dirichlet principle in \eqref{capdef}, with the definition of $\mathcal{V}_{AB}$ in \eqref{VABdef}. 
We assume that $\mathcal{V}_{AB}$ is non-empty or, equivalently, $\Cap(A,B) <\infty$. 

\begin{lemma} 
\label{lem:minimization} 
The Dirichlet form restricted to $\mathcal{V}_{AB}$ has a unique minimizer $h_{AB}$ i.e., there 
is a unique (up to $\mu$-null sets) $h_{AB} \in \mathcal{V}_{AB}$ such that $\Cap(A,B) 
= \mathcal{E}(h_{AB})$. 
\end{lemma} 

\begin{proof}
The lemma is proved via standard convexity, lower semi-continuity and compactness arguments. 
Note that $\mathcal{V}_{AB}$ is a convex set. Since $\phi \mapsto \phi^2$ is strictly convex, we 
have that $\mathcal{E}(t f+ (1-t) g) \leq t\mathcal{E}(f) + (1-t) \mathcal{E}(g)$ for all $f,g\in 
\mathcal{V}_{AB}$ and $t \in (0,1)$, and there is equality if and only if  $f(x) - f(y) = g(x) - g(y)$ 
for $\mu$-a.a.\ $x,y\in \Omega$. The restriction $f|_A = 0 = g|_A$ and the irreducibility of $X$ 
therefore imply that $f = g$ almost everywhere. Hence $\mathcal{E}$ is strictly convex on 
$\mathcal{V}_{AB}$. 

Now let $(f_n)_{n\in\N}$ be a minimizing sequence of functions in $\mathcal{V}_{AB}$ for 
$\mathcal{E}$. Then $(x,y) \mapsto (f_n(x) - f_n(y))_{n\in\N}$  defines a sequence of functions in 
$L^2(\Omega\times \Omega,K)$ that is bounded in $L^2$-norm. The Banach\tire Alaoglu theorem 
ensures the existence of a  subsequence $(x,y) \mapsto (f_{n_j}(x) - f_{n_j}(y))_{j\in \N}$ that 
converges weakly in $L^2(\Omega \times \Omega, K)$, i.e., there is a  function $H \in L^2(\Omega
\times \Omega,K)$ such that
\begin{align} 
&\lim_{j\to \infty} \int_{\Omega\times\Omega} \bigl[ f_{n_j}(x) - f_{n_j}(y) \bigr] 
G(x,y)K(\dd x, \dd y) \\
&\qquad = \int_{\Omega\times\Omega} H(x,y) G(x,y) K(\dd x, \dd y)
\qquad \forall\,G\in L^2(\Omega\times \Omega,K).
\end{align}
The limit function $H$ inherits the following properties (all statements are up to $K$-null sets):
\begin{itemize} 
\item 
$H(x,y) = 1$ on $A\times B$. 
\item 
$H(x,y) = 0$ on $A\times A$ and $B\times B$.
\item 
$H(x,y) \geq 0$ on $A\times \Omega$ and $\Omega \times B$.
\item $H(x,y) + H(y,z) = H(x,z)$ almost everywhere.	
\end{itemize} 
Let $b\in B$ be an arbitrary reference point, and set $h(x)= H(x,b)$. Because of the above properties 
and $\mu(B)>0$, we can choose $b$ such that the following hold: $h=1$ on $A$, $h=0$ on $B$, 
$0 \leq h \leq 1$, and  $H(x,y) = h(x)-h(y)$ almost everywhere, i.e., $h\in \mathcal{V}_{AB}$. Since 
the $L^2$-norm is lower semi-continuous with respect to weak convergence (see Lieb and 
Loss~\cite{lieb-loss}[Theorem 2.11]), we have 
\be
\Cap(A,B) \leq \mathcal{E}(h) \leq \liminf_{j\to \infty} \mathcal{E}(f_{n_j}) 
= \inf_{f \in \mathcal{V}_{AB}} \mathcal{E}(f) = \Cap(A,B),
\ee
so $\mathcal{E}(h) = \Cap(A,B)$ and $h$ is a minimizer. Because of the strict convexity of $\mathcal{E}$,  
it follows that $h_{AB}$ is the unique minimizer. 
\end{proof} 

\begin{lemma} 
\label{lem:harmonic-function} 
{\rm (1)} The minimizer $h_{AB}$ solves the Dirichlet problem.\\
{\rm (2)} If $X$ is recurrent, then $h_{AB}(x) = \P_x( \tau_A < \tau_B)$.\\
{\rm (3)} If $X$ is transient, then $h_{AB}$ is different from the minimal solution $x\mapsto 
\P_x(\tau_A<\zeta,\, \tau_A<\tau_B)$ of the Dirichlet problem. 
\end{lemma} 

\begin{proof}
(1) Suppose by contradiction that $h_{AB}$ does not solve the Dirichlet problem. Then we can find a 
function $f\in L^2(\Omega,\mu)$ such that $f =0$ on $A\cup B$ and $\int_\Omega (-L h_{AB}) f \dd \mu 
< 0$. Set $F(x,y)= f(x)$, $x,y\in\Omega$. Then $F \in L^2(\Omega\times \Omega,K)$ Cauchy\tire Schwarz 
ensures that the integral $\int_{\Omega \times \Omega} f(x) \bigl[ h_{AB}(x) - h_{AB}(y) \bigr] K(\dd x,\dd y)$ 
is absolutely convergent, and
\be
\begin{aligned} 
\mathcal{E}(h_{AB} + \eps f) 
&= \mathcal{E}(h_{AB} ) + \eps \int_{\Omega\times\Omega} f(x) 
\bigl[ h_{AB}(x) - h_{AB}(y) \bigr] K(\dd x, \dd y) + \eps^2 \mathcal{E}(f) \\
&=  \mathcal{E}(h_{AB}) + \eps \int_{\Omega} f(x) \bigl[-Lh_{AB} \bigr](x) \mu(\dd x)  
+ \eps^2 \mathcal{E}(f).
\end{aligned} 
\ee
Choosing $\eps$ small enough we obtain that $\mathcal{E}(h_{AB}+\eps f) < \mathcal{E}(h_{AB})$, 
which is a contradiction. 

\medskip\noindent
(2) If $\P_x(\sigma_{A\cup B} <\infty) =1$ for $\mu$-a.a.\ $x$, then Proposition~\ref{prop:dirichlet-problem} 
implies that the minimizer is equal to the unique solution of the Dirichlet problem $h_{AB}(x) = \P_x(\tau_A
<\tau_B)$.

\medskip\noindent	
(3) The bijection $\mathcal{V}_{AB} \mapsto \mathcal{V}_{BA}$, $h\mapsto 1-h$ leaves the Dirichlet 
form unchanged, because $\mathcal{E}(h) = \mathcal{E}(1-h)$. It follows that $\Cap(A,B) = \Cap(B,A)$ 
and, because the minimizer is unique, 
\be 
\label{eq:abba}
h_{AB} (x) = 1- h_{BA}(x) \quad \text{ for } \mu\text{-a.a. } x.  
\ee
On the other hand, the minimal solution $g_{AB}(x) = \P_x( \tau_A<\zeta,\ \tau_A<\tau_B)$ satisfies 
\be
\label{eq:gab} 
\begin{aligned}  
1 - g_{BA}(x) &= 1 - \P_x\bigl(\tau_B<\zeta,\,\tau_B<\tau_A\bigr) \\
&= \P_x\bigl(\tau_A<\zeta,\,\tau_A<\tau_B\bigr) 
+ \P_x\bigl(\tau_{A \cup B} \geq \zeta\bigr) \\
&= g_{AB}(x) + 1 - \P_x\bigl( \tau_{A\cup B}<\zeta\bigr).
\end{aligned} 
\ee
It follows that $1-g_{BA} \neq g_{AB}$ for transient $X$, and in view of \eqref{eq:abba} we have 
$g_{AB} \neq h_{AB}$.
\end{proof} 


\subsection{Truncation approximation}
\label{ss:truncation}

Next we show that the capacity can be obtained as the limit of certain truncated capacities. This 
will help us prove the Berman\tire Konsowa principle for transient and null recurrent $X$, and provide 
the missing step to show that $\Cap(A,B) = \int_A \bigl( - L h_{AB}) \dd \mu$.

Throughout the sequel we assume that $K((A \cup B) \times \Omega) < \infty$. Let $(\Omega_n)_{n\in \N}$ 
be an increasing sequence of measurable subsets of $\Omega$ with $\cup_{n\in\N}\Omega_n = \Omega$ 
such that $A \cup B \subset \Omega_n$ and $K(\Omega_n \times \Omega_n)<\infty$ for all $n\in\N$ 
\footnote{When $A \cup B \subset \Omega_n$ fails, extend $\Omega_n$ to $\Omega_n'=\Omega_n \cup (A \cup B)$
and note that $K(\Omega_n' \times \Omega_n') < \infty$ because $K(\Omega_n \times \Omega_n) < \infty$
and $K((A \cup B) \times \Omega) < \infty$.}.
Define the truncated kernel $k_n(x,\dd y) = \mathbf{1}_{\Omega_n\times \Omega_n}(x,y) k(x,\dd y)$, 
$x,y\in\Omega$, and the truncated measure $K_n$ as 
\be
K_n(C \times D) = K\bigl( (C \cap \Omega_n) \times (D \cap \Omega_n) \bigr) 
= \int_{C\times D} \mu(\dd x) k_n(x,\dd y), \qquad C,D\subset\Omega. 
\ee
Set
\be
\mathcal{E}_n(f) = \tfrac12 \int_{\Omega \times \Omega} \bigl[ f(x) - f(y) \bigr]^2 K_n(\dd x,\dd y) 
= \tfrac12 \int_{\Omega_n \times \Omega_n} \bigl[ f(x) - f(y) \bigr]^2 K(\dd x,\dd y). 
\ee
The truncated kernel $k_n$ is associated with a reversible jump process for which jumps outside 
$\Omega_n$ are suppressed. 

\begin{remark} 
The theory of Mosco convergence (see Mosco~\cite{mosco}) can be used to show that for 
non-explosive processes the truncated semi-groups and resolvents converge to those of
the full process (see Barlow, Bass, Chen and Kassmann~\cite{barlow-bass-chen-kassmann} 
for references). For our purpose, however, it will be enough to check that the truncated 
capacities converge. 
\end{remark} 

Let $\Cap_n(A,B)$ be the capacity of the truncated process and $h_{AB}^{n}$ the 
corresponding minimizer.

\begin{lemma}[Convergence of truncated capacities]
\label{lem:truncation} 
Suppose that $$K((A \cup B) \times \Omega) <\infty .$$ Then the following hold:\\ 
{\rm (1)} $n \mapsto \Cap_n(A,B)$ is non-decreasing and converges to $\Cap(A,B)$.\\ 
{\rm (2)} $(x,y) \mapsto h_{AB}^n(x) - h_{AB}^n(y)$ converges weakly in $L^2(\Omega\times\Omega,K)$
as $n\to\infty$, i.e.,
\begin{align} 
\label{hnhlim}
&\lim_{n\to \infty}  \int_{\Omega_n\times\Omega_n} \bigl[ h_{AB}^{n}(x) - h_{AB}^{n}(y) \bigr] 
G(x,y) K(\dd x,\dd y) \\ \nonumber
&\qquad =  \int_{\Omega \times\Omega} \bigl[ h_{AB}(x) - h_{AB}(y) \bigr] G(x,y) K(\dd x,\dd y)
\qquad \forall\,G\in L^2(\Omega\times\Omega,K). 		
\end{align}
\end{lemma} 

\begin{proof}
(1) We have $\mathcal{E}_n(f) \leq \mathcal{E}_{n+1}(f) \leq \mathcal{E}(f)$ for all $f\in \mathcal{V}_{AB}$. 
Hence 
\be 
\label{eq:cap-monotonicity} 
\Cap_n(A,B) \leq \Cap_{n+1}(A,B) \leq \Cap(A,B).
\ee
Note that $\mathcal{E}_n(h_{AB}^n) = \Cap_n(A,B) \leq \Cap(A,B)<\infty$. It follows that the sequence 
of functions $(x,y) \mapsto (H_n(x,y))_{n\in \N}$ given by 
\be
H_n(x,y) = \bigl[ h_{AB}^n(x) - h_{AB}^n(y) \bigr] \mathbf{1}_{\Omega_n}(x) \mathbf{1}_{\Omega_n}(y),
\qquad x,y \in \Omega,  
\ee 	
is bounded in $L^2(\Omega\times\Omega,K)$. As in the proof of Lemma~\ref{lem:minimization}, the 
Banach\tire Alaoglu theorem and the lower semi-continuity of the $L^2$-norm with respect to weak 
convergence show that, upon passing to a subsequence, we may assume that $H_n(x,y)$ converges
weakly in $L^2(\Omega \times \Omega,K)$ to $h(x) - h(y)$ for some $h\in\mathcal{V}_{AB}$, and  
\begin{align}
\Cap(A,B) &\leq \tfrac12 \int_{\Omega\times \Omega} \bigl[ h(x) - h(y) \bigr]^2 K(\dd x,\dd y) \nonumber \\ 
&\leq \liminf_{n\to \infty} \tfrac12 \int_{\Omega\times \Omega} H_n(x,y) ^2 K(\dd x,\dd y)
= \lim_{n\to \infty} \Cap_n(A,B). \label{eq:Hbound}
\end{align} 
Together with the inequality in \eqref{eq:cap-monotonicity}, this implies that $\lim_{n\to\infty} \Cap_n(A,B) 
= \Cap(A,B)$ and $\mathcal{E}(h) = \Cap(A,B)$,  and hence $h = h_{AB}$.\\
(2) We see from \eqref{eq:Hbound} that any accumulation point of $H_n(x,y)$ equals $h_{AB}(x)
- h_{AB}(y)$, which implies \eqref{hnhlim}.   
\end{proof} 

\begin{lemma} 
\label{lem:charges} 
The equilibrium charges defined by
\be
Q_A  = \int_{A} \bigl( - L h_{AB}\bigr)(x) \mu(\dd x), \quad  
Q_B = \int_{B} \bigl( - L h_{AB}\bigr)(x) \mu(\dd x),
\ee
satisfy $Q_A = - Q_B = \Cap(A,B)$. 	 
\end{lemma} 

\begin{proof} 
Write
\be
\begin{aligned} 
\mathcal{E}(h_{AB}) &= \lim_{n\to \infty} \tfrac12 \int_{\Omega_n \times \Omega_n} 
\bigl[ h_{AB}^n(x) - h_{AB}^n(y) \bigr] \bigl[ h_{AB}(x) - h_{AB}(y) \bigr]  K(\dd x,\dd y) \\
&= \lim_{n\to \infty} \int_{\Omega_n \times \Omega_n} h_{AB}(x) \bigl[ h_{AB}^n(x) - h_{AB}^n(y) \bigr] 
K(\dd x,\dd y)  \\
& = \lim_{n\to \infty} \int_{\Omega_n} h_{AB}(x) \bigl( - L^n h_{AB}^n\bigr) (x) \mu(\dd x) \\
&= \lim_{n\to\infty} \int_A \bigl( - L^n h_{AB}^n \bigr)(x) \mu(\dd x)
= \int_A (-Lh_{AB})(x) \mu(\dd x), 
\end{aligned} 
\ee
where $L^n$ is the generator of the truncated process on $\Omega_n$. The first equality uses
\eqref{hnhlim} with $G(x,y) = h_{AB}(x)-h_{AB}(y)$, the fourth equality use that $h_{AB}^n$ and 
$h_{AB}$ solve the Dirichlet problem in \eqref{Diriproblem}, the fifth equality uses \eqref{hnhlim}
with $G(x,y) = 1_A(x)$. Thus we have shown that $\Cap(A,B) = \mathcal{E}(h_{AB})= Q_A$. For 
$Q_B$ we apply \eqref{eq:abba} and note that
\be 
Q_B = \int_{B} \bigl( - L h_{AB} \bigr) \dd\mu 
= \int_{B} \bigl( L h_{BA} \bigr) \dd \mu = -\Cap(B,A) = - \Cap(A,B)
\ee
(see the proof of Lemma~\ref{lem:harmonic-function}). 
\end{proof}


\subsection{Thomson principle} 
\label{ss:TPproved}

Let $h_{AB}$ be the unique minimizer in the Dirichlet principle and $\Phi_{AB}$ the associated flow. 

\begin{prop}[Thomson principle] 
For $A,B\subset\Omega$ disjoint, 
\be 
\frac{1}{\Cap(A,B)}  = \min_{\Phi \in \mathcal{U}_{AB}^K} 
\left[ \int_{\Omega \times \Omega} \frac{\dd\Phi}{\dd K}(x,y) \right]^2 K(\dd x,\dd y) 
\ee
and the harmonic flow $\Phi_{AB}$ in \eqref{eq:harmonic-flow} is the unique minimizer. 
\end{prop} 

\begin{proof}
For finite $K$, we have already checked in Lemma~\ref{lem:hf} that $\Phi_{AB}$ is a unit flow. For 
infinite $K$ the proof is similar. The conditions $\Phi_{AB} (A\times\Omega) = \Phi_{AB}(\Omega \times B) 
= 1$ follow from Lemma~\ref{lem:charges}. 
	
Let $h_n=h_{AB}^{n}$ be the truncated harmonic function introduced in Section~\ref{app1:DT}. 
By Lemma~\ref{lem:truncation}, upon passing to a subsequence, we may assume that $h_n(x) 
- h_n(y) \to h_{AB}(x) - h_{AB}(y)$ weakly in $L^2(\Omega\times \Omega, K)$ as $n\to\infty$. 
Let $\Phi$ be a unit flow, and let
\be
\phi(x,y) = \frac{\dd \Phi}{\dd K}(x,y) - \frac{\dd \Phi}{\dd K}(y,x)
\ee
be the antisymmetrized Radon-Nikod{\'y}m derivative of $\Phi$ with respect to $K$. Since the flow 
is directed, we have
\be
\phi(x,y) = \begin{cases} 	
\frac{\dd \Phi}{\dd K}(x,y),
&\quad (x,y) \in \chi,\\
-\frac{\dd \Phi}{\dd K}(y,x),
&\quad (y,x) \in \chi,\\
0, 
&\quad \text{otherwise}, 
\end{cases}
\ee
for some measurable $\chi \subset \Omega\times \Omega$ chosen such that the cases are mutually 
exclusive (see item (4) in Def.~\ref{def:flow}). Note that 
\be
\mathcal{E}(\Phi) = \int_{\Omega \times \Omega }
\left[ \frac{\dd \Phi}{\dd K}(x,y) \right]^2 K(\dd x,\dd y) 
= \frac{1}{2}\int_{\Omega \times \Omega} \phi(x,y)^2 K(\dd x,\dd y).  
\ee
By a slight abuse of notation we use the same letter $\mathcal{E}$ for the quadratic forms on flows 
and on functions. Assume that $\Phi$ has finite energy, i.e., $\phi \in L^2(\Omega\times \Omega, K)$. Let $h_{AB}^n(x)$ be as in the proof of Lemma~\ref{lem:truncation} and $h_n(x):= h_{AB}^n(x) \mathbf{1}_{\Omega_n}(x)$. We have $h_n(x) - h_n(y) \in L^2(\Omega \times \Omega,K)$ and $h_n \in L^2(\Omega,\nu)$ 
with $\nu(\dd x) = \lambda(x) \mu(\dd x)$ the marginal of $K$. By exploiting the symmetry of $K$ and 
the anti-symmetry of $\phi$, we get
\be 
\label{eq:thomsonproof1}
\begin{aligned} 
\tfrac12  \int_{\Omega \times \Omega} \bigl[ h_n(x) - h_n(y) \bigr] \phi(x,y) K(\dd x,\dd y) 
&=  \int_{\Omega \times \Omega}  h_n(x)  \phi(x,y)  K(\dd x,\dd y)\\
& = \int_{\Omega \times \Omega} h_n(x) M_\Phi(\dd x),  
\end{aligned} 
\ee
where $M_\Phi(\dd x)$ is the anti-symmetrized marginal, i.e., $M_\Phi(C) = \Phi(C\times \Omega)
- \Phi(\Omega\times C)$, $C\subset\Omega$ measurable. Since $\Phi$ is a unit flow, we have 
$M_\Phi(C) =0$ for all $C\subset\Omega \backslash (A\cup B)$ and $M_\Phi(A) = 1$. It follows that 
\be 
\int_{\Omega \times \Omega} h_n(x) M_\Phi(\dd x)  = M_\Phi(A) = 1
\ee
and, taking the limit $n\to \infty$ in \eqref{eq:thomsonproof1}, we obtain
\be 
\tfrac12 \int_{\Omega \times \Omega} \bigl[ h_{AB}(x) - h_{AB}(y) \bigr] \phi(x,y) K(\dd x,\dd y) = 1.
\ee	
The  Cauchy\tire Schwarz inequality yields
\be 
[ \Cap(A,B) ]^{1/2} [ \mathcal{E}(\Phi) ]^{1/2} \geq 1
\ee 
with equality if and only if $\phi(x,y) = c (h_{AB}(x)- h_{AB}(y))$ for some $c>0$ and $K$-a.a.\ $(x,y)$. 
The unit flow condition fixes the constant as $c = 1/\Cap(A,B)$, and so
\be
\label{eq:hflow}
\frac{\dd \Phi}{\dd K}(x,y) = \phi(x,y)_+ - \phi(x,y)_ - 
= \frac{1}{\Cap(A,B)} \bigl[ h_{AB}(x) - h_{AB}(y) \bigr]_+.
\ee
Thus, we have found that $\mathcal{E}(\Phi ) \geq 1/\Cap(A,B)$ for all unit flows, with equality if and only 
if $\Phi= \Phi_{AB}$.
\end{proof}


\section{Appendix~B. Physical interpretation of variational principles}
\label{app2}

The Dirichlet principle and the Thomson principle have well-known physical interpretations 
in terms of electric networks (see e.g.\ Doyle and Snell~\cite{doyle-snell}, Peres~\cite{peres},
Gaudilli\`ere \cite{gaudilliere}). In Section~\ref{app2:DT} we recall these interpretations, while 
in Section~\ref{app2:BK} we give an interpretation of the Berman\tire Konsowa principle that 
is based on three ingredients: (1) resistances in series add up; (2) conductances in parallel 
add up; (3) each network has an equivalent network that is richer but simpler, consisting 
of chains of resistors in parallel. The latter interpretation, which was suggested in 
Bovier~\cite{bovier}, is worked out in detail. 


\subsection{Dirichlet and Thomson}
\label{app2:DT}

For the sake of exposition we assume that $\Omega$ is \emph{finite}, and write $K_{xy}$ for 
$K(\{(x,y)\})$. Let $G$ be the undirected graph with vertex set $\Omega$ and edge 
set $E'=\{\{x,y\} \in \Omega \times \Omega \mid K_{xy}>0\}$. For later purpose, we write 
$E$ for the corresponding set of directed edges, i.e., $E = \{(x,y)\in \Omega \mid K_{xy}>0\}$. 
The network consists of a set of resistors: each edge $\{x,y\}$ of $G$ has resistance $R_{xy} 
= 1/K_{xy}$. Recall \emph{Ohm's law}: voltages $(V_{x})_{x \in \Omega}$ induce currents 
$(i_{xy})_{(x,y) \in E}$. The current flows from high to low voltage, and the intensity of the 
current is $K_{xy} |V_x - V_y|$. We adopt the convention that $i_{xy} =0$ when $V_x<V_y$,
so that $i_{xy} = K_{xy} (V_x - V_y)_+$.\footnote{Another convention is to take $i_{xy}<0$ 
when $V_x<V_y$, so that the sign of $i$ carries the information of the direction of the current.} 
Furthermore, by the \emph{Joule effect}, the current through the resistor network dissipates 
energy at a rate 
\be
\label{eq:dissipation}
\mathcal{P} = \sum_{(x,y) \in E} i_{xy} \bigl( V_x - V_y\bigr)_+ 
= \tfrac{1}{2} \sum_{x,y\in \Omega} K_{xy} \bigl(V_x - V_y\bigr)^2
=  \sum_{(x,y)\in E} R_{xy} i_{xy}^2, 
\ee
which is the \emph{electric power} (measured in Watt). The factor $\tfrac12$ can be removed 
when we replace $(V_x - V_y)$ by $(V_x- V_y)_+$. In \eqref{eq:dissipation} we recognize the 
Dirichlet form $\mathcal{E}(h)$ in \eqref{eq:dirichlet-form} applied to the function $h(x)= V_x$. 

Now, let $A$ and $B$ be two disjoint subsets of $\Omega$. We extend the network by adding 
two points to $\Omega$: the \emph{source} $s$ and the \emph{sink} $s'$. The edge set is 
enriched by connecting all points from $A$ to $s$ and all points from $B$ to $s'$. The new 
edges are assigned resistance zero (infinite conductivity). As a consequence, the vertices in 
$A$  and $B$ have the same voltage as the source and sink (``wiring''). Moreover, the flow 
through the new edges does not dissipate energy, so that the total energy dissipation is still 
given by \eqref{eq:dissipation}. 

It is a standard problem from electrical engineering to determine the \emph{effective resistance} 
$R_\text{eff}$ (or effective conductance $C_\text{eff} = 1/R_\text{eff}$) between the source and 
the sink, defined as follows. Fix $V_s$ and $V_{s'}$, the voltages of the source and the sink,
assume $V_s>V_{s'}$, and set $U = V_s - V_{s'}$. Let 
\be
\bigl( (V_x)_{x\in \Omega},(i_{xy})_{\{x,y\} \in E}\bigr)
\ee 
solve the following set of equations:
\begin{itemize} 
\item 
Wiring to source and sink: 
$V_a = V_s$ for all $a\in A$, $V_b = V_{s'}$ for all $b\in B$. 
\item 
Ohm's law: $i_{xy} = K_{xy} (V_x - V_y)_+$ for all $\{x,y\} \in E$. 
\item 
Kirchhoff's law:
\be \label{eq:kirchhoff}
i_{sz} + \sum_{ {x\in \Omega\colon} \atop {\{x,z\} \in E} } i_{xz} 
= i_{zs'} +   \sum_{ {y\in \Omega\colon} \atop {\{z,y\} \in E} } 
i_{zy} \qquad \forall\,z \in \Omega. 
\ee
\end{itemize}
For most networks there will be a unique solution, which satisfies $V_{s'} \leq 
V_x \leq V_{s}$ for all $x \in \Omega$. Ohm's law implies that nothing flows into the 
source (or into $A$) and nothing flows out of the sink (or out of $B$). The total flow out 
of the source is
\be 
I = \sum_{a \in A} i_{sa} = \sum_{a \in A} \sum_{y \in \Omega} i_{ay}.
\ee
On finite networks, Kirchhoff's law implies that the current out of the source equals the 
current into the sink, i.e.,
\be 
I = \sum_{b\in B} i_{b s'} = \sum_{b\in B} \sum_{x \in \Omega} i_{xb}. 
\ee
The effective resistance is 
\be
\label{eq:ohm-eff}
R_\text{eff} = \frac{U}{I} = \frac{V_{s} - V_{s'}}{I},
\ee
which depends on the sets $A$ and $B$, and on the conductances $K_{xy}$, but not on 
the voltages $V_s$, $V_{s'}$. Therefore we can evaluate the effective resistance for 
$V_s =1$, $V_{s'} =0$, in which case the voltage distribution $(V_x)_{x\in \Omega}$ is 
equal to the harmonic function $h_{AB}(x)$ in \eqref{eq:habdef}, and the ``capacity'' in
\eqref{eq:cap-def1} is equal to 
\be
\Cap(A,B) =  - \sum_{a\in A} \sum_{y\in \Omega}  K_{a y} (V_y - V_a) 
= \sum_{a \in A} \sum_{y\in \Omega} i_{a y} = I. 
\ee
Comparing this expression with \eqref{eq:ohm-eff} and remembering our choice $U = V_{s} 
- V_{s'} = 1-0 =1$, we see that 
\be
\Cap(A,B) = \frac{I}{U} = \frac{1}{R_\text{eff}} = C_\text{eff}, 
\ee
i.e., our mathematical capacity is equal to the effective conductance of the network. 

\begin{remark}
The use of the word ``capacity'', though legitimate from a mathematical point of view, is not 
quite appropriate in the physical context of electric networks. Indeed, the current-voltage 
relation of a capacitor is $I(t) = C \dd U(t)/\dd t $ with $C>0$ the capacity: this is 
clearly inconsistent with our relation  $I = C_\text{eff} U$. The word ``capacity'' becomes 
legitimate when the Dirichlet form is interpreted as an electrostatic energy rather than a 
dissipation rate -- but in this context the ``conductances'' must be replaced by ``dielectric 
permittivities''. The probabilist vocabulary is a hybrid of two distinct physical pictures, and 
the physicist reader must be aware of this potential source of confusion. 
\end{remark}

The electric power in \eqref{eq:dissipation}, evaluated for the voltage and current distribution 
solving the set of equations described above, equals
\be
\mathcal{P} =U I = C_\text{eff} U^2 = R_\mathrm{eff} I^2.   
\ee
The Dirichlet principle says that minimization of the electric power over all voltage distributions 
with net voltage $U= V_s- V_{s'}$ yields the effective conductance, while the Thompson principle 
says that minimization of the power over all current distributions with net current $I=1$ yields the 
effective resistance. 


\subsection{Berman\tire Konsowa}
\label{app2:BK}

The interpretation of the Berman\tire Konsowa principle is more involved. We start with two 
simple examples. 

The first example consists of a finite chain of resistors \emph{in series},  e.g., $\Omega = 
\{0,1,\ldots,N\}$, $A=\{0\}$, $B=\{N\}$, $K_{xy} >0$ for $|x-y|=1$ and $K_{xy} =0$ otherwise. 
Resistances in series add up, hence $R_\mathrm{eff} = \sum_{j=1}^{N} R_{j-1,j}$ 
or, equivalently, 
\be
C_\mathrm{eff} = \left( \sum_{j=1}^{N} \frac{1}{K_{j-1,j}}\right)^{-1}.
\ee
We recognize the right-hand side of \eqref{eq:Cap1} from Theorem~\ref{thm:BK1} evaluated 
for the deterministic path $\gamma=(0,1,\ldots,N)$. 

The second example consists of resistors \emph{in parallel}. Suppose that $\Omega = A \cup B$, 
so that we may think of the network as a set of resistors $R_{ab}$ in parallel between the source 
and the sink. Conductances in parallel add up, hence 
\be
C_\text{eff} = \sum_{a\in A} \sum_{b\in B} K_{ab}.  
\ee
On the other hand, let $(p_{ab})_{a\in A,b\in B}$ be such that $\sum_{a\in A,b \in B} p_{ab} =1$ 
and $p_{ab}>0$ for all $a\in A$, $b\in B$. The weights $p_{ab}$ determine a probability measure 
on ``paths'' of length $1$ from $A$ to $B$. The right-hand side of \eqref{eq:Cap1} equals 
\be
\sum_{a\in A,\,b\in B} p_{ab} \Bigl( \frac{p_{ab}}{K_{ab}} \Bigr)^{-1} 
=  \sum_{a\in A,\,b\in B}  K_{ab}. 
\ee

Hence the Berman\tire Konsowa principle reproduces the additivity of resistances (in series) or 
conductances (in parallel). Now consider an arbitrary finite network and fix a probability measure 
$\P$ on $\Gamma_{AB}$, the set of finite paths from $A$ to $B$. Let $\Phi_\P$ be the associated 
flow. We construct a new network as follows: 

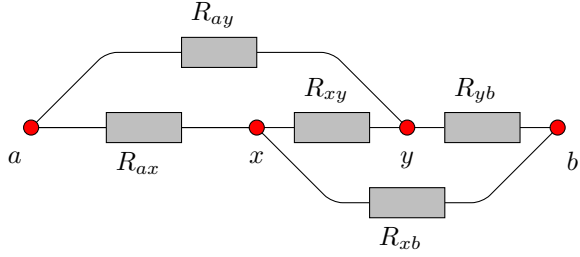
\begin{figure}
\begin{tikzpicture} [>=triangle 45]
	\draw [rounded corners] (0,1)--(1,2)--(4,2)--(5,1);
	\draw (0,1)--(3,1)--(5,1)--(7,1);
	\draw [rounded corners] (0,1)--(3,1)--(4,0)--(6,0)--(7,1); 
	\draw[fill=red] (0,1)  circle [radius =.1];
	\draw[fill=red] (3,1)  circle [radius =.1];
	\draw[fill=red] (5,1)  circle [radius =.1];
	\draw[fill=red] (7,1)  circle [radius =.1];
	\draw [fill=lightgray] (2,2.2) rectangle (3,1.8);
	\draw [fill=lightgray] (1,1.2) rectangle (2,0.8);
	\draw [fill=lightgray] (3.5,1.2) rectangle (4.5, 0.8);
	\draw [fill=lightgray] (5.5,1.2) rectangle (6.5,0.8);
	\draw [fill=lightgray] (4.5,0.2) rectangle (5.5,-0.2);
	\node [below left] at (0,.8) {$a$};
	\node [below] at (3,.8) {$x$}; 
	\node[below] at (5,.8) {$y$}; 
	\node[below right] at (7,.8) {$b$};
	\node[above right] at (2,2.2) {$R_{ay}$};
	\node[below right] at (1,.8) {$R_{ax}$};
	\node[above right] at (3.5,1.2) {$R_{xy}$};
	\node[above right] at (5.5,1.2) {$R_{yb}$};
	\node[below right] at (4.5,-0.2) {$R_{xb}$};
	\node[right, text width = 5cm] at (8,1) 
	{(a) The original network has four vertices $a,x,y,b$. Each edge $e=(ij)$ 
	is associated with a resistance $R_{ij}$.}; 

	\begin{scope}[yshift=-4cm]
	\draw [->,  rounded corners] (0,1)--(1,2)--(4,2)--(4.9,1.1);
	\draw[->, rounded corners] (0,1)--(2.9,1);
	\draw[->, rounded corners] (3,1)-- (4.9,1);
	\draw [->,  rounded corners] (5,1)--(6.9,1);
	\draw[->,  rounded corners] (3,1)--(4,0)--(6,0)--(6.9,.9); 
	\draw[fill=red] (0,1)  circle [radius =.1];
	\draw[fill=red] (3,1)  circle [radius =.1];
	\draw[fill=red] (5,1)  circle [radius =.1];
	\draw[fill=red] (7,1)  circle [radius =.1];
	\node [below left] at (0,.8) {$a$};
	\node [below] at (3,.8) {$x$}; 
	\node[below] at (5,.8) {$y$}; 
	\node[below right] at (7,.8) {$b$};
	\node[above right] at (2,2) {$\Phi(a,y)$};
	\node[below right] at (1,1) {$\Phi(a,x)$};
	\node[above right] at (3.5,1) {$\Phi(x,y)$};
	\node[above right] at (5.5,1) {$\Phi(y,b)$};
	\node[below right] at (4.5,0) {$\Phi(x,b)$}; 
	\node[right, text width = 5cm] at (8,1) 
	{(b) An $ab$-unit flow $\Phi(i,j)$ determines a stochastic matrix $\ell(i,j)$ 
	and a probability measure $\P$ on paths from $a$ to $b$. E.g., $\ell(a,x)
	= \Phi(a,x)$, $\ell(x,b)= \Phi(x,b) / M(x)$ with $M(x) = \Phi(a,x)$, and 
	$\P(a\to x \to b) = \ell(a,x) \ell(x,b)$.} ;
	\end{scope} 

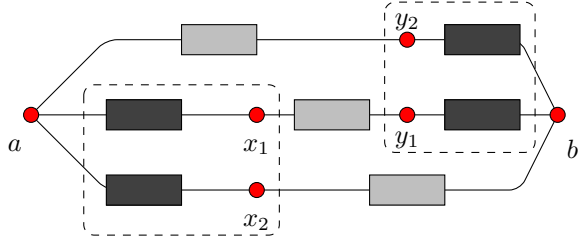
	\begin{scope}[yshift=-8cm]
	\draw [rounded corners] (0,1)--(1,2)--(6.5,2)--(7,1);
	\draw (0,1)--(3,1)--(5,1)--(7,1);
	\draw [rounded corners] (0,1)--(1,0)--(6.5,0)--(7,1); 
	\draw[fill=red] (5,2)  circle [radius =.1];
	\draw[fill=red] (0,1)  circle [radius =.1];
	\draw[fill=red] (3,1)  circle [radius =.1];
	\draw[fill=red] (5,1)  circle [radius =.1];
	\draw[fill=red] (7,1)  circle [radius =.1];
	\draw[fill=red] (3,0)  circle [radius =.1];
	\draw [fill=lightgray] (2,2.2) rectangle (3,1.8);
	\draw [fill=darkgray] (1,1.2) rectangle (2,0.8);
	\draw [fill=lightgray] (3.5,1.2) rectangle (4.5, 0.8);
	\draw [fill=darkgray] (5.5,1.2) rectangle (6.5,0.8);
	\draw [fill=lightgray] (4.5,0.2) rectangle (5.5,-0.2);
	\draw [fill=darkgray] (5.5,2.2) rectangle (6.5,1.8);
	\draw [fill=darkgray] (1,0.2) rectangle (2,-0.2);
	\node [below left] at (0,.8) {$a$};
	\node [below] at (3,.8) {$x_1$}; 
	\node[below] at (5,0.9) {$y_1$}; 
	\node[below right] at (7,.8) {$b$};
	\node[below] at (3,-0.2) {$x_2$}; 
	\node[above] at (5,2) {$y_2$};
	\draw [dashed, rounded corners ](0.7,1.4) rectangle (3.3,-0.6);
	\draw[dashed,rounded corners](4.7,2.5) rectangle (6.7, 0.5); 
	\node[right, text width = 5cm] at (8,1) 
	{(c) The new network consists of parallel chains of resistances; 
	each chain comes from a path $\gamma$ allowed under $\P$.  
	An edge or vertex visited by more than one path is split accordingly 
	(dashed boxes).};
	\end{scope}	
\end{tikzpicture} 
\caption{\label{fig} Electric network constructed out of a unit flow: 
(a) the original network, (b) a unit flow, (c) the associated new network.} 
\end{figure} 


\begin{enumerate} 
\item 
With each path $\gamma \in \Gamma$, associate a chain of resistors in series with reduced resistances 
$R_{xy}^\gamma =  R_{xy}\Phi_\P(x,y)/\P(\gamma)$, $(x,y) \in \gamma$. The effective 
conductance of this chain is 
\be 
C_\text{eff}^\P(\gamma) 
=\left( \sum_{(x,y) \in \gamma} R_{xy}^\gamma \right)^{-1} 
= \P(\gamma)  \left( \sum_{(x,y) \in \gamma} \frac{\Phi_\P(x,y)}{K_{xy}} \right)^{-1}.
\ee
\item 
Put the chains in parallel, so that they do not overlap. As a consequence, the equivalent network 
in general has more vertices than the original network (Figure~\ref{fig}). The total 
conductance is 
\be
C_\text{eff}^\P = \sum_{\gamma\in \Gamma} C_\text{eff}^\P(\gamma). 
\ee	
\end{enumerate}



\end{document}